\documentclass[12pt]{amsart}
\usepackage{amsmath,amssymb,amsxtra,amscd}
\usepackage[margin=3cm]{geometry} 
\usepackage[hypertex]{hyperref}
\usepackage[all]{xy}

\numberwithin{equation}{section}

\newtheorem{theorem}{Theorem}[section]
\newtheorem{lemma}[theorem]{Lemma} 
\newtheorem{proposition}[theorem]{Proposition} 
\newtheorem{corollary}[theorem]{Corollary} 
\theoremstyle{definition}
\newtheorem{definition}[theorem]{Definition} 
\newtheorem{notation}[theorem]{Notation} 
\newtheorem{remark}[theorem]{Remark}

\newcommand{\C}{\mathbb{C}}
\newcommand{\Cstar}{\C^\times} 
\newcommand{\Z}{\mathbb{Z}} 
\newcommand{\Q}{\mathbb{Q}} 
\newcommand{\R}{\mathbb{R}}
\newcommand{\bS}{\mathbb{S}} 
\newcommand{\tbS}{\widetilde{\bS}} 
\newcommand{\PP}{\mathbb{P}} 
\newcommand{\LL}{\mathbb{L}} 

\newcommand{\cO}{\mathcal{O}} 
\newcommand{\cS}{\mathcal{S}} 
\newcommand{\tcS}{\widetilde{\cS}} 
\newcommand{\cE}{\mathcal{E}} 
\newcommand{\cF}{\mathcal{F}} 
\newcommand{\cN}{\mathcal{N}}
\newcommand{\cL}{\mathcal{L}} 
\newcommand{\ovcL}{\overline{\cL}}
\newcommand{\cH}{\mathcal{H}} 

\newcommand{\hT}{{\widehat{T}}} 
\newcommand{\htau}{\hat{\tau}}
\newcommand{\hd}{\hat{d}} 
\newcommand{\hGamma}{\widehat{\Gamma}} 
\newcommand{\hu}{\hat{u}}

\newcommand{\ttau}{\tilde{\tau}} 
\newcommand{\tUpsilon}{\widetilde{\Upsilon}} 

\newcommand{\bx}{\mathbf{x}} 
\newcommand{\bN}{\mathbf{N}}
\newcommand{\bt}{\mathbf{t}}
\newcommand{\bbf}{\mathbf{f}}
\newcommand{\bV}{\mathbf{V}} 

\newcommand{\frs}{\mathfrak{s}} 
\newcommand{\frI}{\mathfrak{I}}

\newcommand{\iu}{\mathtt{i}} 

\newcommand{\vir}{{\rm vir}} 
\newcommand{\mov}{{\rm mov}}
\renewcommand{\sec}{{\rm sec}}  
\newcommand{\loc}{{\rm loc}} 

\newcommand{\Spec}{\operatorname{Spec}} 
\newcommand{\Proj}{\operatorname{Proj}} 
\newcommand{\id}{\operatorname{id}} 
\newcommand{\ev}{\operatorname{ev}} 
\newcommand{\pt}{\operatorname{pt}} 
\newcommand{\Frac}{\operatorname{Frac}} 
\newcommand{\Lie}{\operatorname{Lie}} 
\newcommand{\Eff}{\operatorname{Eff}} 
\newcommand{\PD}{\operatorname{PD}} 
\newcommand{\Hom}{\operatorname{Hom}} 
\newcommand{\Ext}{\operatorname{Ext}} 
\newcommand{\End}{\operatorname{End}} 
\newcommand{\rank}{\operatorname{rank}}
\newcommand{\ch}{\operatorname{ch}} 
\newcommand{\Ker}{\operatorname{Ker}}

\def\corr#1{\left\langle#1 \right\rangle} 
\def\parfrac#1#2{\frac{\partial #1}{\partial #2}} 

\begin{document} 

\title{Shift operators and toric mirror theorem}
\author{Hiroshi Iritani} 
\email{iritani@math.kyoto-u.ac.jp} 
\address{Department of Mathematics, Graduate School of Science, 
Kyoto University, Kitashirakawa-Oiwake-cho, Sakyo-ku, 
Kyoto, 606-8502, Japan}
\begin{abstract} 
We give a new proof of Givental's mirror theorem for toric manifolds 
using shift operators of equivariant parameters. 
The proof is almost tautological: it gives an A-model construction of the $I$-function 
and the mirror map. It also works for non-compact or 
non-semipositive toric manifolds. 
\end{abstract} 

\maketitle

\section{Introduction} 
In 1995, Seidel \cite{Seidel:pi1} introduced an invertible element 
of quantum cohomology associated to a Hamiltonian circle 
action. This has had many applications in symplectic topology. 
Seidel himself used it to construct non-trivial 
elements of $\pi_1$ of the group of Hamiltonian diffeomorphisms. 
McDuff-Tolman \cite{McDuff-Tolman} calculated Seidel's elements 
in a more general setting and obtained Batyrev's ring presentation of 
quantum cohomology of toric manifolds. Their method, however, 
does not yield explicit structure constants of quantum cohomology, 
i.e.~genus-zero Gromov-Witten invariants. 

Recently, Braverman, Maulik, Okounkov and Pandharipande  
\cite{Okounkov-Pandharipande:Hilbert, 
BMO:Springer, Maulik-Okounkov:qcoh_qgroup} 
introduced a shift operator of equivariant parameters 
for equivariant quantum cohomology. 
Their shift operators reduce to Seidel's invertible elements 
under the non-equivariant limit. 
In this paper, we show that equivariant genus-zero 
Gromov-Witten invariants of toric manifolds are reconstructed 
\emph{only from formal properties of shift operators}. 
This means that the equivariant quantum topology of toric manifolds is  
determined by its classical counterpart.

More specifically, we give a new proof of Givental's mirror theorem 
for toric manifolds, which is stated as follows: 
\begin{theorem}[\cite{Givental:toric_mirrorthm,LLY:III,
Iritani:genmir,Brown:toric_fibration}, see 
\S \ref{subsec:mirrorthm} for more details]
Let $X_\Sigma$ be a semi-projective toric manifold having 
a torus fixed point. 
Let $I(y,z)$ be the cohomology-valued hypergeometric series defined 
by 
\[
I(y,z) = z e^{\sum_{i=1}^m u_i \log y_i/z} 
\sum_{d\in \Eff(X_\Sigma)} 
\left(\prod_{i=1}^m 
\frac{\prod_{c=-\infty}^0(u_i + cz)}{\prod_{c=-\infty}^{u_i \cdot d} (u_i + cz)}
\right) 
Q^d y_1^{u_1\cdot d}\cdots y_m^{u_m \cdot d} 
\] 
where $u_i$, $i=1,\dots,m$ is the class of a prime toric divisor. 
Then $I(y,-z)$ lies in Givental's Lagrangian cone $\cL_{X_\Sigma}$ 
associated to $X_\Sigma$. 
\end{theorem}

We prove this theorem in the following way. 
Recall that equivariant genus-zero Gromov-Witten invariants of 
a $T$-variety $X$ can be encoded by an infinite-dimensional 
Lagrangian submanifold $\cL_X$ of the symplectic 
vector space $\cH_X$ \cite{Givental:symplectic}:
\[
\cH_X = H^*_T(X)\otimes_{H_T^*(\pt)} \Frac\left(H_T^*(\pt)[z]\right). 
\]
The space $\cH_X$ is called the Givental space and 
$\cL_X$ is called the Givental cone. 
By the general theory, each $\C^\times$-subgroup $k \colon \C^\times \to T$ 
defines a shift operator $\cS_k$ acting on the Givental space $\cH_X$ 
and induces a vector field on $\cL_X$: 
\[
\cL_X \ni \bbf \longmapsto z^{-1} \cS_k\bbf \in T_\bbf \cL_X.  
\]
The operator $\cS_k$ is determined by $T$-fixed loci in $X$
and their normal bundles  (see Definition \ref{def:shift_Givental}). 
For toric manifolds, we have a shift operator $\cS_i$ for each 
torus-invariant prime divisor. Then we identify 
the $I$-function $I(y,z)$ with 
an integral curve of the commuting vector fields $\bbf \mapsto z^{-1} \cS_i \bbf$. 
\begin{theorem} 
Givental's $I$-function $I(y,z)$ is a unique integral curve 
which satisfies the differential equation: 
\[
\parfrac{I(y,z)}{y_i} = z^{-1} \cS_i I(y,z) 
\qquad \qquad 
i=1,\dots,m 
\]
and is of the form $
I(y,z)=z e^{\sum_{i=1}^m u_i \log y_i/z} 
(1 + \sum_{d\in \Eff(X_\Sigma)\setminus\{0\}} 
I_d Q^d y^d)$, where we set $y^d = \prod_{i=1}^m 
y_i^{u_i\cdot d}$. 
\end{theorem} 

The $I$-function defines a mirror map $y\mapsto \tau(y) 
\in H^*_T(X)$ via Birkhoff factorization \cite{Coates-Givental,Iritani:genmir}. 
As a corollary to our proof, we obtain the following relationship between 
the equivariant Seidel elements $S_i(\tau)$ and the mirror map. 
This generalizes a previous result \cite{Gonzalez-Iritani:Selecta} 
in the semipositive case obtained in joint work with Gonzalez. 
\begin{theorem} 
The mirror map $\tau(y)$ associated to the $I$-function is a 
unique integral curve which satisfies the differential equation 
\[
\parfrac{\tau(y)}{y_i} = S_i(\tau(y))  \qquad \qquad 
i=1,\dots,m 
\]
and is of the form 
$\tau(y) = \sum_{i=1}^m u_i \log y_i 
+ \sum_{d\in \Eff(X_\Sigma)\setminus \{0\}} \tau_d Q^d y^d$.  
\end{theorem} 
The mirror map and the $I$-function are related by 
the formula 
\[
I(y,z) = z M(\tau(y),z) \Upsilon(y,z)
\]
where $M(\tau,z)$ is a fundamental solution for 
the quantum differential equation (Proposition \ref{prop:fundsol}) 
and $\Upsilon(y,z)$ is an $H^*_T(X)[z]$-valued function. 
We can also characterize $\Upsilon(y,z)$ by the differential 
equation 
\[
\parfrac{\Upsilon(y,z)}{y_i} = [z^{-1} \bS_i(\tau(y))]_+\Upsilon(y,z) 
\]
where $\bS_i(\tau)$ is the shift operator acting on quantum cohomology. 
The most technical point in our proof 
is to show the existence of solutions $\tau(y)$ and $\Upsilon(y,z)$ 
with prescribed asymptotics 
(see Proposition \ref{prop:tau_Upsilon}). 

Since we do not assume that $c_1(X_\Sigma)$ is nef, the mirror map 
$\tau(y)$ does not necessarily lie in $H^{\le 2}_T(X)$. 
For this reason, we need to generalize shift operators to big quantum cohomology. 
We also observe that shift operators are closely related to 
the $\hGamma$-integral structure \cite{Iritani, KKP, CIJ}. 
We show that a flat section of the quantum 
connection associated to an equivariant vector bundle 
in the formalism of $\hGamma$-integral structure 
is invariant under shift operators (Proposition \ref{prop:Gamma_integral_shift}).  

This paper is structured as follows. 
In \S \ref{sec:qcoh}, we review equivariant 
quantum cohomology and in \S \ref{sec:shift}, 
we study shift operators for big quantum cohomology. 
In \S \ref{sec:mirrorthm}, we prove a mirror theorem 
for toric manifolds.

\subsection{Notation} 
Unless otherwise stated, we consider cohomology groups 
with complex coefficients. We use the following notation 
throughout the paper. 
\begin{itemize} 
\item $T\cong (\Cstar)^m$: an algebraic torus;  
\item $X$: a smooth $T$-variety; 
$X_\Sigma$: a smooth toric variety associated to a fan $\Sigma$; 
\item $\hT = T\times \Cstar$; 
\item $\lambda\in \Lie(T)$, 
$z\in \Lie(\Cstar)$: equivariant parameters for $\hT$;  
\item $H_\hT(X)_{\loc} := 
H_\hT^*(X)\otimes_{H_\hT^*(\pt)} \Frac(H_\hT^*(\pt)) 
= H_T^*(X)\otimes_{H_T^*(\pt)} \Frac(H_T^*(\pt)[z])$: 
the Givental space.  
\end{itemize} 

\medskip 
\noindent 
{\bf Acknowledgments.} 
The author thanks Tom Coates, Alessio Corti, Eduardo Gonzalez, 
Hiraku Nakajima and Hsian-Hua Tseng for very helpful 
discussions on shift operators, Seidel representations 
and toric mirror symmetry. This work is supported by 
JSPS Kakenhi Grant Number 25400069. 

\section{Equivariant Quantum Cohomology} 
\label{sec:qcoh} 
\subsection{Hypotheses on a $T$-Space} 
\label{subsec:hypotheses} 
Let $T\cong (\Cstar)^m$ be an algebraic torus. 
Let $X$ be a smooth variety over $\C$ 
equipped with an algebraic $T$-action. 
We assume the following conditions: 
\begin{enumerate} 
\item 
\label{semiproj} $X$ is semi-projective, i.e.~the natural map 
$X\to X_0 :=\Spec H^0(X,\cO)$ is projective; 
\item  \label{weight_cone} 
all $T$-weights appearing in the $T$-representation 
$H^0(X,\cO)$ are contained 
in a strictly convex cone in $\Hom(T,\Cstar)\otimes \R$ 
and $H^0(X,\cO)^T = \C$. 
\end{enumerate} 
A $T$-space $X$ with these assumptions 
has nice cohomological properties, 
see, e.g.~\cite{HaRV:large}. 
These conditions ensure that the $T$-fixed set $X^T$ is projective. 
We also note the following:  
\begin{proposition} 
\label{prop:equiv_formal}
A smooth $T$-variety $X$ satisfying the conditions {\rm (\ref{semiproj}), 
(\ref{weight_cone})} 
is equivariantly formal, i.e.~$H_T^*(X)$ is a free module 
over $H_T^*(\pt)$ and there is a non-canonical isomorphism 
$H_T^*(X) \cong H^*(X) \otimes H_T^*(\pt)$ as an 
$H_T^*(\pt)$-module. 
\end{proposition} 
\begin{proof} 
We use the argument of Kirwan 
\cite[Proposition 5.8]{Kirwan:coh_quotient} 
(see also \cite[\S 5.1]{Nakajima:lectures}). 
Choose a one-parameter subgroup $k \colon \Cstar \to T$ 
such that $k$ is negative on every non-zero weight of 
$H^0(X,\cO)$. This defines a $\Cstar$-action on $X$. 
Let $L\to X$ be a very ample line bundle. 
The $\Cstar$-action on $X$ lifts to a $\Cstar$-linearization 
on $L$, possibly after replacing $L$ with its power $L^{\otimes i}$ 
\cite[Corollary 7.2]{Dolgachev:invariant}. 
Then $L$ defines a $\Cstar$-equivariant closed embedding 
$X \hookrightarrow X_0 \times \PP^n$, where $\PP^n$ 
is equipped with a linear $\Cstar$-action. 
By assumption, we can embed the affine variety 
$X_0 = \Spec (H^0(X,\cO))$ equivariantly 
into a $\Cstar$-representation $V$ which 
has only positive\footnote{We use the (usual) convention that 
$t\in \Cstar$ acts on \emph{functions} by 
$f(x) \mapsto f(t^{-1} x)$.} weights. 
Thus we have a $\Cstar$-equivariant closed embedding 
$X \hookrightarrow V \times \PP^n$. 
The associated $S^1$-action on $V\times \PP^n$ admits, 
with respect to the standard K\"{a}hler metric, 
a moment map $\mu$ which is proper and bounded from below. 
These properties allow us to use Morse theory for the moment map 
$\mu|_X$. The argument in 
\cite{Kirwan:coh_quotient, Nakajima:lectures} shows that  
$\mu|_X$ is a perfect Bott-Morse function and 
$X$ is equivariantly  formal. 
\end{proof} 

\subsection{Gromov-Witten Invariants} 
For a second homology class $d\in H_2(X,\Z)$ and a non-negative 
integer $n\ge 0$, we denote by $X_{0,n,d}$ the moduli 
stack of genus-zero stable maps to $X$ of degree $d$ 
with $n$ marked points. The $T$-action on $X$ induces 
a $T$-action on $X_{0,n,d}$. 
It has a virtual fundamental 
class $[X_{0,n,d}]_{\rm vir}\in H_*(X_{0,n,d},\Q)$ 
of dimension $D = \dim X + n -3 + c_1(X) \cdot d$. 
For equivariant cohomology classes $\alpha_1,\dots,\alpha_n 
\in H^*_T(X,\Q)$ and non-negative integers $k_1,\dots,k_n$, 
the genus-zero $T$-equivariant Gromov-Witten invariant is defined by 
\[
\corr{\alpha_1 \psi^{k_1},\dots, \alpha_n \psi^{k_n}}_{0,n,d}^{X,T} 
= \int_{[X_{0,n,d}]_{\rm vir}} \prod_{i=1}^n \ev_i^*(\alpha_i) \psi_i^{k_i}. 
\]
Here $\ev_i \colon X_{0,n,d} \to X$ is the evaluation map at the 
$i$th marked point and $\psi_i$ denotes the equivariant first Chern class 
of the $i$th universal cotangent line bundle $L_i$ over $X_{0,n,d}$. 
When the moduli space $X_{0,n,d}$ is not compact, the right-hand side 
is defined via the Atiyah-Bott localization formula 
\cite{Atiyah-Bott, Graber-Pandharipande} and belongs to 
the fraction field $\Frac(H_T^*(\pt))$ of $H_T^*(\pt)$. 

\subsection{Quantum Cohomology} 
Let $\Eff(X) \subset H_2(X,\Z)$ denote the semigroup of 
homology classes of effective curves. 
We write $Q$ for the Novikov variable and define 
$M[\![Q]\!]$ to be the space of formal power series: 
\[
M[\![Q]\!] = \left\{ 
\textstyle\sum_{d\in \Eff(X)} a_d Q^d : a_d \in M\right\} 
\]
with coefficients in a module $M$. 
When $M$ is a ring, $M[\![Q]\!]$ is also a ring.  
Let $(\cdot,\cdot)$ denote the $T$-equivariant Poincar\'{e} pairing 
on $H_T^*(X)$: 
\[
(\alpha,\beta) = \int_X \alpha \cup \beta.  
\]
If $X$ is not compact, we define the right-hand side via 
the localization formula. Therefore 
$(\cdot,\cdot)$ takes values in $\Frac(H_T^*(\pt))$ in general. 
Let $\{\phi_i\}_{i=0}^N$ be a basis of $H^*_T(X)$ over $H_T^*(\pt)$. 
We write $\{\tau^i\}_{i=0}^N$ for the dual co-ordinates on $H^*_T(X)$ 
and $\tau = \sum_{i=0}^N \tau^i\phi_i$ for a general point on $H^*_T(X)$. 
The (big) quantum product $\star$ is defined by the formula 
\[
(\phi_i \star \phi_j,\phi_k) = 
\sum_{d\in \Eff(X)} \sum_{n=0}^\infty \frac{Q^d}{n!}
\corr{\phi_i,\phi_j,\phi_k,\tau,\dots,\tau}_{0,n+3,d}^{X,T}. 
\] 
We note that the quantum product $\phi_i\star\phi_j$ 
is defined without localization: 
\[
\phi_i \star \phi_j \in H_T^*(X)[\![Q]\!][\![\tau^0,\dots,\tau^N]\!]. 
\]
In fact, $\phi_i\star \phi_j$ can be written as the push-forward 
\begin{equation} 
\label{eq:qprod_another}
\sum_{d\in \Eff(X)} \sum_{n=0}^\infty 
\frac{Q^d}{n!} \PD \ev_{3*} 
\left ( \ev_1^*(\phi_i) \ev_2^*(\phi_j) \prod_{l=4}^{n+3} 
\ev_l^*(\tau) \cap [X_{0,n+3,d}]_{\rm vir}\right) 
\end{equation} 
along the \emph{proper} evaluation map $\ev_{3}$, and 
hence the localization is not necessary. 
The properness of $\ev_3$ follows from the assumption that 
$X$ is semi-projective. 

\subsection{Quantum Connection and Fundamental Solution} 
The quantum connection is the operator 
\[
\nabla_i \colon H_T^*(X)[z][\![Q]\!][\![\tau^0,\dots,\tau^N]\!] 
\to z^{-1} H_T^*(X)[z][\![Q]\!][\![\tau^0,\dots,\tau^N]\!] 
\]
defined by 
\[
\nabla_i = \parfrac{}{\tau^i} + \frac{1}{z} (\phi_i \star). 
\]
The quantum connection has a parameter $z$: we identify 
it with the equivariant parameter for an additional $\Cstar$-action. 
Set $\hT = T\times \Cstar$ and consider the $\hT$-action 
on $X$ induced by the projection $\hT \to T$. 
Then we have $H^*_\hT (X) \cong H^*_T(X) [z]$. 
The quantum connection is known to be flat, and admits a 
fundamental solution: 
\[
M(\tau) \colon H_\hT^*(X)[\![Q]\!][\![\tau^0,\dots,\tau^N]\!] \to 
H_\hT^*(X)_{\loc}[\![Q]\!][\![\tau^0,\dots,\tau^N]\!] 
\]
satisfying the quantum differential equation: 
\[
z \parfrac{}{\tau^i} M(\tau) = M(\tau) (\phi_i\star)
\]
or equivalently $(\partial/\partial \tau^i)\circ M(\tau) 
= M(\tau) \circ \nabla_i$,  
where $H_\hT^*(X)_{\loc} := H_\hT^*(X)\otimes_{H_\hT^*(\pt)} 
\Frac(H_\hT^*(\pt))$ is the localized equivariant cohomology. 
The following proposition is well-known, see 
\cite[\S 1]{Givental:elliptic}, 
\cite[Proposition 2]{Pandharipande:afterGivental}. 
\begin{proposition} 
\label{prop:fundsol} 
A fundamental solution is given by 
\[
(M(\tau) \phi_i,\phi_j) = (\phi_i, \phi_j) +  
\sum_{\substack{d\in \Eff(X), n\ge 0 \\ (d,n) \neq (0,0)}}
\frac{Q^d}{n!} 
\corr{\phi_i,\tau,\dots,\tau,\frac{\phi_j}{z-\psi}}_{0,n+2,d}^{X,T}. 
\]
\end{proposition} 
\begin{remark} 
Expanding $1/(z-\psi) = \sum_{n=0}^\infty \psi^n/z^{n+1}$, we find 
that $M(\tau)\phi_i$ takes values in $H_T^*(X)[\![z^{-1}]\!]$. 
By the localization calculation, it also follows that $M(\tau) \phi_i$ takes 
values in $H_{\hT}^*(X)_{\loc}$. 
The localized $\hT$-equivariant cohomology $H_\hT^*(X)_{\loc}$ 
is also called the \emph{Givental space} \cite{Givental:symplectic}. 
\end{remark} 

\section{Shift Operator} 
\label{sec:shift} 
The shift operator for equivariant quantum cohomology has been  
introduced by Okounkov-Pandharipande \cite{Okounkov-Pandharipande:Hilbert}, 
Braverman-Maulik-Okounkov \cite{BMO:Springer} and 
Maulik-Okounkov \cite{Maulik-Okounkov:qcoh_qgroup}. 
We discuss its (straightforward) extension to the big quantum cohomology. 

\subsection{Twisted Homomorphism} 
We write $\hT = T \times \Cstar$. 
For a group homomorphism $k\colon \Cstar \to T$, 
we consider the $\hT$-action $\rho_k$ on $X$ 
defined by 
\[
\rho_k(t,u) x = t u^k \cdot x 
\]
where $(t,u) \in \hT$, $x\in X$ and 
$u^k\in T$ denotes the image of $u\in \Cstar$ under $k$. 
Let $\lambda \in \Lie(T)$ denote the equivariant 
parameter for $T$ and let $z\in \Lie(\Cstar)$ denote 
the equivariant parameter for $\Cstar$. 
The identity map $\id \colon (X,\rho_0) \to (X,\rho_k)$ 
is equivariant with respect to the group automorphism 
\[
\phi_k \colon \hT \to \hT, \qquad 
\phi_k(t,u) = (t u^{-k}, u).
\]
Therefore the identity map induces an isomorphism 
\[
\Phi_{k} \colon 
H^*_{\hT, \rho_0}(X) \cong H^*_{\hT,\rho_k}(X) 
\]
such that 
\begin{equation} 
\label{eq:twisted_hom} 
\Phi_k(f(\lambda,z) \alpha) = f(\lambda+kz, z) \Phi_k(\alpha) 
\end{equation} 
where $\alpha \in H_{\hT,\rho_0}^*(X)$ and 
$f(\lambda,z) \in H_{\hT}^*(\pt)$ is a polynomial 
function on $\Lie(\hT)$. 
Referring to the property \eqref{eq:twisted_hom}, 
we say that $\Phi_k$ is a \emph{$k$-twisted} homomorphism. 

\begin{notation}
We write $H_{\hT,\rho}^*(X)$ for the $\hT$-equivariant cohomology 
of $X$ with respect to the $\hT$-action $\rho$ on $X$. When $\rho$ is 
omitted, $H_{\hT}^*(X)$ means $H_{\hT,\rho_0}^*(X)$. 
\end{notation} 

\subsection{Bundle Associated to a $\Cstar$-Subgroup} 

\begin{definition}[associated bundle] 
Let $k \colon \C^\times \to T$ be a group homomorphism. 
Consider the $\C^\times$-action on $X \times (\C^2\setminus \{0\})$ 
given by $s \cdot (x,(v_1,v_2)) = (s^k \cdot x, (s^{-1} v_1, s^{-1} v_2))$. 
Let $E_k$ denote the quotient space: 
\[
E_k := X \times (\C^2 \setminus \{0\}) / \Cstar.  
\]
We have a natural projection $\pi \colon E_k \to \PP^1$ given by 
$\pi ([x,(v_1,v_2)] )= [v_1,v_2]$ and 
$E_k$ is a fiber bundle over $\PP^1$ with fiber $X$. 
We consider the $\hT$-action 
on $E_k$ given by $
(t,u) \cdot [x,(v_1,v_2)] = [t \cdot x, (v_1, u v_2)]$. 
Let $X_0$ denote the fiber of $E_k \to \PP^1$ 
at $[1,0]$ and let $X_\infty$ denote the fiber at $[0,1]$. 
Note that we have 
\[
X_0 \cong (X,\rho_0) \qquad \text{and} \qquad 
X_\infty \cong (X,\rho_k) 
\]
as $\hT$-spaces. 
\end{definition} 

\begin{definition} 
\label{def:seminegative} 
A group homomorphism $k \colon \Cstar \to T$ is said to be 
\emph{semi-negative} if $k$ is non-positive on each $T$-weight 
of $H^0(X,\cO)$. We say that $k$ is \emph{negative} if $k$ is negative 
on each non-zero $T$-weight of $H^0(X,\cO)$. 
\end{definition} 

\begin{remark}
When $X$ is complete, every $\Cstar$-subgroup is negative. 
\end{remark} 

Suppose that $k\colon \Cstar \to T$ is semi-negative  
and consider the $\Cstar$-action on $X$ induced by $k$. 
Let $L$ be a very ample line bundle on $X$. As discussed 
in the proof of Proposition \ref{prop:equiv_formal}, 
we may assume that $L$ admits a $\Cstar$-linearization. 
By tensoring $L$ with a $\Cstar$-character, 
we may assume that all the $\Cstar$-weights on 
$H^0(X,L^{\otimes n})$ are negative for $n>0$.  
Let $p\colon X \times \C^2 \to X$ be the natural 
projection. 
Then $p^*L$ is a $\Cstar$-equivariant line 
bundle on $X \times \C^2$, where $\Cstar$ acts 
on the base by $s \cdot (x, (v_1,v_2)) = 
(s^k\cdot x, (s^{-1} v_1, s^{-1} v_2))$.  
We can see that 
\[
H^0(X\times \C^2, (p^*L)^{\otimes n}) = 
\bigoplus_{i=0}^\infty H^0(X,L^{\otimes n})^{(-i)} 
\otimes \C[v_1,v_2]^{(i)} 
\]
where the superscript $(l)$ means the component of 
$\Cstar_s$-weight $l$. 
The unstable locus for the $\Cstar$-action on $(X\times \C^2,p^*L)$, 
in the sense of Geometric Invariant Theory (GIT), 
is $X\times \{0\}$ and therefore we find that $E_k$ is the GIT quotient of 
$X\times \C^2$, i.e.~$E_k = \Proj(\bigoplus_{n=0}^\infty 
H^0(X\times \C^2, (p^*L)^{\otimes n}))$. 
This proves: 

\begin{lemma} 
\label{lem:E_semiproj} 
If $k$ is semi-negative, $E_k$ is semi-projective. 
\end{lemma}

Let $k\colon \Cstar \to T$ be a semi-negative subgroup 
and consider the $\Cstar$-action on $X$ induced by $k$. 
A $\Cstar$-fixed point $x\in X$ defines a section of $E_k 
\to \PP^1$: 
\begin{equation} 
\label{eq:fixed_point_section} 
\sigma_x = ( \{ x\} \times \PP^1 ) \subset E_k.
\end{equation} 
We now define a minimal section among all such sections 
associated to fixed points. 
Using the argument in the proof of Proposition 
\ref{prop:equiv_formal}, we obtain a $\Cstar$-equivariant closed 
embedding $X\hookrightarrow \PP^n \times \C^l$ where 
$\C^l$ is a $\Cstar$-representation with only non-negative weights. 
In particular, for every point $x\in X$, the limit $\lim_{s\to 0} s^k \cdot x$ 
exists. This implies the existence of the Bialynicki-Birula  
decomposition \cite[Theorem 4.1]{Bialynicki-Birula} for $X$: 
if $X^{\Cstar}= \bigsqcup_i F_i$ is the decomposition 
of the $\Cstar$-fixed locus $X^{\Cstar}$ into connected components, 
we have the induced decomposition of $X$ 
\[
X = \bigsqcup_i U_i, \qquad 
U_i = \left\{x\in X: \lim_{s \to 0} s^k \cdot x \in F_i\right\} 
\]
into locally closed smooth subvarieties $U_i$. 
In particular there exists a unique $\Cstar$-fixed 
component $F_{\min}\subset X$ such that all the $\Cstar$-weights 
on the normal bundle to $F_{\min}$ are positive. The moment map 
$\mu$ for the associated $S^1$-action attains a global minimum on $F_{\min}$. 
We call the class of a section $\sigma_{\min}$ of $E_k$ 
associated to a point in $F_{\min}$ 
the \emph{minimal section class}. 
We write 
\begin{align*} 
H_2^\sec(E_k,\Z) 
& = \left\{ d\in H_2(E_k,\Z) : \pi_*(d) = [\PP^1] \right\}, \\ 
\Eff(E_k)^\sec & = \Eff(E_k) \cap H_2^\sec(E_k,\Z).
\end{align*} 
\begin{lemma}
If $k$ is semi-negative, we have 
$\Eff(E_k)^\sec = \sigma_{\min} + \Eff(X)$. 
\end{lemma} 
\begin{proof} 
The compact case was discussed in \cite[Lemma 2.2]{Gonzalez-Iritani:Selecta}. 
Take a negative one-parameter subgroup $l\colon \Cstar \to T$ and 
consider the $\Cstar$-action on $E_k$ induced by 
$\Cstar \xrightarrow{l} T \times \{1\} \subset \hT$. 
Observe that all non-zero $\Cstar$-weights on $H^0(E_k,\cO)$ are negative. 
This means that $E_{k,0} :=\Spec H^0(E_k,\cO)$ has 
a unique $\Cstar$-fixed point $0$ and $\lim_{s\to 0} s\cdot x = 0$ 
for all $x\in E_{k,0}$. Therefore every curve can be 
deformed, via the $\Cstar$-action, to a stable curve in the fiber $K$ 
of $E_k \to E_{k,0}$ at $0\in E_{k,0}$ in the same homology class. 
Since $\hT$-action on $E_k$ preserves $K$ and $K$ is compact, 
we may further deform a curve in $K$ to a 
$\hT$-invariant stable curve. 
A $\hT$-invariant stable curve in $E_k$ is a union of 
a section class $\sigma_x$ associated to a $T$-fixed point $x\in X$ 
and effective curves in $X_0 \sqcup X_\infty$. 
Suppose that two different fixed points $x, y \in X^T$ 
are connected by a $k(\Cstar)$-orbit, i.e.~$\exists p\in X$, 
$x = \lim_{s \to \infty} s^k \cdot p$ and 
$y = \lim_{s \to 0} s^k \cdot p$. 
The closure $C= \overline{k(\Cstar)\cdot p}$ is 
isomorphic to $\PP^1$ and $\sigma_x,\sigma_y$ are contained 
in a Hirzebruch surface 
\[
C \times (\C^2\setminus \{0\})/\Cstar 
\subset E_k. 
\] 
Then one finds $\sigma_x = \sigma_y + a [C]$ 
for some $a>0$. 
Using the Bialynicki-Birula decomposition for the 
$k(\Cstar)$-action on $X$, we find that every $T$-fixed point 
is connected to a $T$-fixed point on $F_{\min}$ 
by a chain of $k(\Cstar)$-orbits. The conclusion follows. 
\end{proof}

\begin{lemma} 
\label{lem:coh_E} 
We have an isomorphism 
\begin{align*} 
H_{\hT}^*(E_k) & \cong 
\left\{(\alpha, \beta) \in H_{\hT,\rho_0}^*(X) \oplus 
H_{\hT,\rho_k}^*(X) : \alpha - \Phi_k^{-1}(\beta)\equiv 0 \mod z\right\} 
\end{align*} 
which sends $\tau$ to $(\tau|_{X_0}, \tau|_{X_\infty})$.  
Recall that $z$ is the equivariant 
parameter for $\Cstar$ and we have a canonical isomorphism 
$H_{\hT,\rho_0}^*(X) \cong H_T^*(X) [z]$.   
\end{lemma} 
\begin{proof} 
Consider the Mayer-Vietoris exact sequence associated 
to the covering 
$E_k =U_0 \cup U_\infty$ with $U_0=\pi^{-1}(\C)$ and 
$U_\infty =\pi^{-1}(\PP^1\setminus \{0\})$. 
We have $H_{\hT}^*(U_0) \cong H_{\hT,\rho_0}^*(X)$, 
$H_{\hT}^*(U_\infty) 
\cong H_{\hT,\rho_k}^*(X)$ and 
$H_{\hT}^*(U_0\cap U_\infty) \cong H_T(X)$. 
The map $H_{\hT}^*(U_0) \oplus H_\hT^*(U_\infty) 
\to H_\hT^*(U_0\cap U_\infty)$ is surjective and 
is given by $(\alpha,\beta) \mapsto (\alpha - \Phi_k^{-1}\beta)|_{z=0}$. 
\end{proof} 

\begin{notation} 
By Lemma \ref{lem:coh_E}, for $\tau\in H_T^*(X)$, 
there exists $\htau \in H_{\hT}^*(E_k)$ such that 
$\htau|_{X_0} = \tau$ and $\htau|_{X_\infty} = \Phi_k(\tau)$.  
This defines a map $\hat{\phantom{\tau}} \colon 
H_T^*(X) \to H_{\hT}^*(E_k)$. This is not 
$H_T^*(\pt)$-linear. 
\end{notation} 

\subsection{Shift Operator} 
\begin{definition}[shift operator] 
Let $k \colon \Cstar \to T$ be a semi-negative group homomorphism. 
For $\tau \in H_T^*(X)$, 
we define $\tbS_k(\tau) \colon H^*_{\hT,\rho_0}(X)[\![Q]\!] \to 
H^*_{\hT,\rho_k}(X)[\![Q]\!]$ by 
\[
\left( \tbS_k(\tau) \alpha, \beta \right) =
\sum_{\hd\in \Eff(E_k)^\sec} 
\frac{Q^{\hd-\sigma_{\min}}}{n!} 
\corr{\iota_{0*}\alpha,
\iota_{\infty*} \beta, \htau, \dots, \htau}_{0,n+2,\hd}^{E_k,\hT} 
\]
where $(\cdot,\cdot)$ in the left-hand side is the $\hT$-equivariant 
Poincar\'{e} pairing on $H_{\hT,\rho_k}^*(X)$, 
$\alpha \in H_{\hT,\rho_0}^*(X)$, $\beta\in H_{\hT,\rho_k}^*(X)$, 
$\sigma_{\min}$ is the minimal section class for $E_k$, and 
$\iota_0\colon X_0 \to E_k$, $\iota_\infty \colon X_\infty \to E_k$ 
are the natural  inclusions. 
We also define 
\[
\bS_k(\tau) = \Phi_k^{-1} \circ \tbS_k(\tau) \colon 
H_{\hT}^*(X)[\![Q]\!] \to H_\hT^*(X)[\![Q]\!].
\]
Note that $\tbS_k$ is untwisted but 
$\bS_k$ is $(-k)$-twisted (see \eqref{eq:twisted_hom}). 
\end{definition} 

\begin{remark} 
When $k$ is semi-negative, $E_k$ is semi-projective by Lemma \ref{lem:E_semiproj} 
and thus the shift operator $\bS_k$ is defined without localization: 
we may rewrite $\tbS_k$ as the push-forward 
along an evaluation map (see \eqref{eq:qprod_another}). 
When $k$ is not semi-negative, we can still define $\bS_k$ over 
$\Frac(H_T^*(\pt))$ after choosing a suitable section class 
$\sigma_{\min}$. 
\end{remark} 

\begin{remark} 
Since the map $\tau \mapsto \htau$ is not $H_T^*(\pt)$-linear, 
$\bS(\tau)$ cannot be written as formal power series in the 
$H_T^*(\pt)$-valued variables $\tau^0,\dots,\tau^N$. 
For $\alpha_1,\dots,\alpha_l\in H_T^*(X)$ and 
\emph{$\C$-valued} variables $t^1,\dots,t^l$, the shift operator 
$\bS(\tau)$ with $\tau = \sum_{i=1}^l t^i \alpha_i$ is a formal 
power series in $t^1,\dots,t^l$. 
\end{remark} 

\begin{remark}[divisor equation]  
\label{rem:div_shift} 
Suppose that $\tau = h + \tau'$ with $h \in H^2_T(X)$. 
Using the divisor equation, we have: 
\[
\left( \tbS_k(\tau) \alpha, \beta \right ) 
= e^{-h(k)}
\sum_{d \in \Eff(X)}
\frac{Q^d e^{h\cdot d}}
{n!} 
\corr{\iota_{0*} \alpha, \iota_{\infty *} \beta, \htau',
\dots \htau'}
_{0,n+2,\sigma_{\min}+d}^{E_k,\hT} 
\]
where $h(k)$ is the pairing between $k$ and 
the restriction $h|_x\in H^2_T(\pt) \cong \Lie(T)^*$ of 
$h$ to a fixed point $x$ in the minimal fixed component 
$F_{\min}$ (with respect to $k$). 
Note that $\hat{h} \cdot \sigma_{\min} = - h(k)$. 
\end{remark} 

By the localization theorem of equivariant cohomology 
\cite{Atiyah-Bott}, the restriction to the $T$-fixed subspace $X^T$ 
induces an isomorphism 
\[
\iota^* \colon H_\hT^*(X)_{\loc} \overset{\cong}{\longrightarrow} 
H^*_{\hT}(X^T)_{\loc} 
= H^*(X^T) \otimes \Frac(H_\hT^*(\pt)). 
\]
We use this to define the shift operator on the Givental space 
$H_\hT^*(X)_{\loc}$. 

\begin{definition}[shift operator on the Givental space] 
\label{def:shift_Givental}
Let $X^T=\bigsqcup_i F_i$ be the decomposition of $X^T$ 
into connected components. Let $N_{i}$ be the normal bundle 
to $F_i$ in $X$. Let 
$N_{i}= \bigoplus_\alpha N_{i,\alpha}$ 
denote the $T$-eigenbundle decomposition, 
where $T$ acts on $N_{i,\alpha}$ 
by the character $\alpha \in \Hom(T,\Cstar)$. 
Let $\rho_{i,\alpha,j}$, $j=1,\dots,\rank(N_{i,\alpha})$ 
denote the Chern roots of $N_{i,\alpha}$. 
For a semi-negative $k\in \Hom(\Cstar,T)$, we define: 
\[
\Delta_i(k)= Q^{\sigma_i -\sigma_{\min}} 
\prod_{\alpha} \prod_{j=1}^{\rank (N_{i,\alpha}) }
\frac{\prod_{c=-\infty}^0 (\rho_{i,\alpha,j} + \alpha + cz)
}{\prod_{c=-\infty}^{-\alpha \cdot k} (\rho_{i,\alpha,j}+\alpha + cz)} 
\in H_\hT^*(F_i)_{\loc}[\![Q]\!]
\]
where $\alpha$ is regarded as an element of $H^2_T(\pt,\Z)$, 
$\sigma_i$ is the section class of $E_k$ associated to 
a fixed point in $F_i$ and $\sigma_{\min}$ is the minimal section class 
of $E_k$. 
Note that all but finite factors in the infinite product cancel. 
We define the operator $\cS_k \colon H_\hT^*(X)_{\loc} 
\to H_\hT^*(X)_{\loc}$ by the following 
commutative diagram:  
\begin{equation} 
\label{eq:diag_cS}
\begin{aligned} 
\xymatrix{
H_\hT^*(X)_{\loc} \ar[rrr]^{\cS_k} \ar[d]_{\iota^*} 
&&& H_\hT^*(X)_{\loc} \ar[d]_{\iota^*} \\ 
H_\hT^*(X^T)_{\loc} 
\ar[rrr]^{\bigoplus_i \Delta_i(k) e^{-z k \partial_\lambda}} 
&&&  H_\hT^*(X^T)_{\loc}
} 
\end{aligned} 
\end{equation} 
where we use the 
decomposition $H_\hT^*(X^T)_{\loc} \cong \bigoplus_i 
H^*(F_i)\otimes \Frac(H_\hT(\pt))$ in the bottom arrow 
and $e^{-k z\partial_\lambda}$ acts on $\Frac(H_\hT(\pt))$ by 
$f(\lambda,z) \mapsto f(\lambda-kz, z)$. 
The operator $\cS_k$ is a $(-k)$-twisted homomorphism. 
\end{definition} 

The following is a key property of the shift operator. 

\begin{theorem}
\label{thm:intertwine} 
We have $M(\tau) \circ \bS_k(\tau) = \cS_k \circ M(\tau)$, 
where $M(\tau)$ is the fundamental solution in 
Proposition \ref{prop:fundsol}. 
\end{theorem} 
\begin{proof} 
A similar intertwining property has been discussed 
in \cite{Okounkov-Pandharipande:Hilbert, 
BMO:Springer, Maulik-Okounkov:qcoh_qgroup}. 
We calculate $\tbS_k(\tau)$ using $\hT$-equivariant 
localization. We refer the reader to 
\cite{Graber-Pandharipande, Cox-Katz} 
for localization arguments in Gromov-Witten theory. 
Fix a section class $\hd \in  \Eff(E_k)^\sec$. 
A $\hT$-fixed stable map $f \colon (C,x_1,\dots,x_{n+2}) \to E_k$ 
of degree $\hd$ is of the form:  
\begin{itemize} 
\item $C = C_0 \cup C_\sec \cup C_\infty$ with $C_\sec \cong \PP^1$; 
\item $f_0 =f|_{C_0}$ is a $T$-fixed stable map to $X_0$;  
\item $f_\infty = f|_{C_\infty}$ is a $T$-fixed stable map to $X_\infty$;  
\item $f_\sec = f|_{C_\sec}$ is a section of $E_k$ associated to 
a $T$-fixed point in $X$ (see \eqref{eq:fixed_point_section}). 
\end{itemize} 
Recall that the tangent space $T^1$ and the obstruction 
space $T^2$ at the stable map $f$ fit into the exact sequence 
\[
\begin{CD}
0 @>>> \Ext^0(\Omega_C^1(\bx),\cO_C) @>>> H^0(C,f^*T_{E_k}) 
@>>> T^1  \\ 
@>>> \Ext^1(\Omega_C^1(\bx),\cO_C) @>>> H^1(C,f^* T_{E_k}) 
@>>> T^2 
@>>> 0 
\end{CD}
\]
where $\bx = x_1+\cdots + x_{n+2}$. 
The virtual normal bundle at $f$ is: 
\begin{align*} 
\cN^\vir & = T^{1,\mov} - T^{2,\mov} 
= \chi(f^*T_{E_k})^\mov  
 -  \chi(\Omega_C^1(\bx),\cO_C)^\mov 
\end{align*} 
where ``mov'' means the moving part with respect to 
the $\hT$-action and 
$\chi(\cE) = H^0(C,\cE) - H^1(C,\cE)$, 
$\chi(\cE,\cF) = \Ext^0(\cE,\cF) - 
\Ext^1(\cE,\cF)$ denotes the Euler characteristics. 
Let $p$, $q$ denote the nodal intersection 
points $C_0 \cap C_\sec$, $C_\infty \cap C_\sec$ 
respectively. 
Using the normalization exact sequence 
$0 \to \cO_C \to \cO_{C_0} \oplus \cO_{C_\sec} 
\oplus \cO_{C_\infty} \to \C_p \oplus \C_q \to 0$, 
we find: 
\begin{align}
\label{eq:map_deform} 
\begin{split} 
\chi(f^*T_{E_k})^\mov 
& = 
\chi(f_0^*T_{X_0})^\mov + \chi(f_\infty^* T_{X_\infty})^\mov 
+ \chi(f_\sec^*T_{E_k}) \\ 
& + \xi + \xi^{-1}  - (T_{f(p)}E)^\mov - (T_{f(q)}E)^\mov
\end{split} 
\end{align} 
where $\xi$ is the one-dimensional $\Cstar$-representation 
of weight one. 
We write $\bx = \bx_0 + \bx_\infty$ where 
$\bx_0$, $\bx_\infty$ are divisors on $C_0$, $C_\infty$ respectively. 
Then we have 
\begin{align}
\label{eq:curve_deform} 
\begin{split} 
-\chi(\Omega^1_C(\bx),\cO_C)^\mov & = T_p C_0 \otimes T_p C_\sec  + 
T_q C_\infty \otimes T_q C_\sec \\ 
& -\chi(\Omega^1_{C_0}(\bx_0+p),\cO_{C_0})^\mov 
- \chi(\Omega^1_{C_\infty}(\bx_\infty +q), \cO_{C_\infty})^\mov. 
\end{split} 
\end{align} 
The $\hT$-fixed locus in the moduli space 
$(E_k)_{0,n+2,\hd}$ is given by 
\[
\bigsqcup_i \bigsqcup_{I_1 \sqcup I_2 = \{1,\dots,n+2\}} 
\bigsqcup_{d_0+d_\infty+ \sigma_i = \hd} 
((X_0)_{0,I_1\cup p,d_0})^T \times_{F_i}  
((X_\infty)_{0,I_2\cup q ,d_\infty})^T 
\]
where $F_i$, $\sigma_i$ are as in Definition \ref{def:shift_Givental}. 
Combining \eqref{eq:map_deform}, \eqref{eq:curve_deform}, 
we find that the virtual normal bundle $\cN_i^\vir$ 
on the component 
$((X_0)_{0,I_1\cup p,d_0})^T \times_{F_i}  
((X_\infty)_{0,I_2\cup q ,d_\infty})^T$ is:  
\[
\cN^\vir_i = \cN_0^\vir + \cN_\infty^\vir + \cN_{\sec,i}
 - N_{F_i/X_0} - N_{F_i/X_\infty} 
+ L_p^{-1} \otimes \xi + L_q^{-1} \otimes \xi^{-1} 
\]
where $\cN^\vir_0$ is the virtual normal bundle 
of $(X_0)_{0,I_1\cup p,d_0}^T$ in $(X_0)_{0,I_1\cup p, d_0}$, 
$\cN^\vir_\infty$ is the virtual normal bundle 
of $(X_\infty)_{0,I_2\cup q,d_\infty}^T$ in 
$(X_\infty)_{0,I_2\cup q, d_\infty}$, 
$L_p$ (resp.~$L_q$) is the universal cotangent line bundle 
at $p$ (resp.~$q$) and 
$\cN_{\sec,i}$ is the vector bundle with fiber 
$\chi(f_\sec^*T_{E_k})^\mov$. 
Let $N_{F_i/X} = N_i = \oplus_\alpha N_{i,\alpha}$ be decomposition 
as in Definition \ref{def:shift_Givental}. 
The normal bundle of $F_i\times \PP^1$ in $E_k$ is 
\[
\bigoplus_{\alpha} N_{i,\alpha} \boxtimes \cO_{\PP^1}(-\alpha\cdot k). 
\]
Thus we find: 
\begin{equation} 
\label{eq:Nsec} 
\cN_{\sec,i} = \xi \oplus \xi^{-1} \oplus 
\bigoplus_{\alpha} N_{i,\alpha}\otimes 
\left( \bigoplus_{c\le 0} \xi^{c}  - 
\bigoplus_{c<\alpha\cdot k} \xi^c \right).  
\end{equation} 
The virtual localization formula gives: 
\begin{multline*} 
\left(\tbS_k(\tau) \alpha,\beta\right) 
= \sum_{i,k,l,a,b}  \sum_{d_0 + d_\infty + \sigma_i = \hd} 
\corr{z \alpha, \tau,\dots,\tau,\frac{(\iota_{0,i})_*\phi_{i,a}}{z-\psi}}
^{X_0, \hT}_{0,k+2,d_0} \frac{Q^{d_0}}{k!}\\ 
\times 
\left( 
\int_{F_i} 
\frac{Q^{\sigma_i-\sigma_{\min}}}{e_\hT(\cN_{\sec,i})} 
\phi_i^a \phi_i^b \right)  
\corr{\frac{(\iota_{\infty,i})_*\phi_{i,b}}{-z-\psi}, \tau',\dots,\tau', -z \beta}
^{X_\infty,\hT}_{0,l+2,d_\infty}
\frac{Q^{d_\infty}}{l!}  
\end{multline*} 
where $\alpha \in H^*_{\hT}(X_0)$, $\beta \in H^*_{\hT}(X_\infty)$, 
$\tau' = \Phi_k (\tau)$, the maps $\iota_{0,i} \colon F_i \to X_0$, 
$\iota_{\infty,i}\colon F_i \to X_\infty$ are the natural inclusions, 
$\{\phi_{i,a}\}\subset H^*(F_i)$ is a basis, 
$\{\phi_i^a\}$ is the dual basis such that $\int_{F_i} \phi_{i,a} 
\cup \phi_i^b = \delta_a^b$. 
Note that we have by \eqref{eq:Nsec}, 
\[
\frac{Q^{\sigma_i -\sigma_{\min}}}{e_\hT(\cN_{\vir,i})} 
= \frac{1}{z(-z)} \frac{1}{e_\hT(N_{F_i/X_\infty})} 
\left(e^{k z\partial_\lambda} \Delta_i(k)\right).  
\]
Combining these equations, we conclude 
\[
\left(\tbS_k(\tau) \alpha, \beta \right) 
= \left( \tcS_k M(\tau,z) \alpha, M'(\tau',-z) \beta \right)
\]
where we write the argument $z$ in the fundamental 
solution explicitly and 
\begin{itemize} 
\item $\tcS_k \colon H_\hT(X_0)_\loc \to H_\hT(X_\infty)_\loc$ 
is a map defined similarly to $\cS_k$ by replacing 
$\bigoplus_i \Delta_i(k) e^{-k z\partial_\lambda}$ in 
the diagram \eqref{eq:diag_cS} 
with $\bigoplus_i (e^{kz \partial_\lambda} \Delta_i(k))$; 
\item $M'(\tau',z)$ is defined similarly to Proposition 
\ref{prop:fundsol} by replacing $T$-equivariant 
Gromov-Witten invariants there with $(\hT,\rho_k)$-equivariant 
invariants. 
\end{itemize} 
Note that $M'(\tau',z) = \Phi_k \circ M(\tau,z) \circ \Phi_k^{-1}$ 
and $\tcS = \Phi_k \circ \cS$. 
The conclusion follows from 
the so-called ``unitarity'' $M(\tau,-z)^* = M(\tau,z)^{-1}$ 
of the fundamental solution (see \cite[\S 1]{Givental:elliptic}). 
\end{proof} 

Theorem \ref{thm:intertwine} and the differential equation 
$\partial_i \circ M(\tau) = M(\tau) \circ \nabla_i$ show: 

\begin{corollary} 
\label{cor:intertwine}
The shift operator commutes with the quantum connection, 
i.e.~$[\nabla_i, \bS_k(\tau)] = 0$ for $i=0,\dots,N$. 
\end{corollary} 

This corollary is shown in \cite[\S 8]{Maulik-Okounkov:qcoh_qgroup} 
in the case where $\tau=0$. 
We also remark that the shift operators commute each other. 
\begin{corollary} 
\label{cor:commuting} 
We have $\cS_k \circ \cS_l = Q^{d(k,l)}\cS_{k+l}$ 
for some $d(k,l)\in H_2(X,\Z)$ which is symmetric in $k$ and $l$. 
In particular, $\bS_k \circ \bS_l = Q^{d(k,l)} \bS_{k+l}$, 
$[\cS_k,\cS_l] = [\bS_k, \bS_l ] = 0$. 
\end{corollary} 
\begin{proof} 
Consider the $X$-bundle $E_{k,l}$ over $\PP^1\times \PP^1$ 
given by 
\[
E_{k,l} = X \times (\C^2 \setminus \{0\}) \times 
(\C^2\setminus \{0\})\big/\C^\times \times \C^\times 
\]
where $(s_1,s_2)\in \C^\times \times \C^\times$ 
acts on $X \times \C^2 \times \C^2$ 
by $(s_1,s_2) \cdot (x,(a_1,a_2),(b_1,b_2)) = 
(s_1^k s_2^l, (s_1^{-1} a_1, s_1^{-1} a_2), 
(s_2^{-1}b_1, s_2^{-1} b_2))$. Note that $E_{k,l}|_{\PP^1\times [1:0]} 
\cong E_k$ and $E_{k,l}|_{[1:0]\times \PP^1} \cong E_l$ 
and $E_{k,l}|_{\Delta(\PP^1)} \cong E_{k+l}$, where 
$\Delta(\PP^1) \subset \PP^1\times \PP^1$ denotes the diagonal. 
The addition in $H_2(E_{k,l},\Z)$ defines a map 
$\# \colon H_2^\sec(E_l,\Z) \times H_2^\sec(E_k, \Z) \to H_2^\sec(E_{k+l},\Z)$. 
For any $T$-fixed point $x$, the section class $\sigma_x$ 
(see \eqref{eq:fixed_point_section}) associated to $x$ 
satisfies $\sigma_x \# \sigma_x = \sigma_x$. 
A straightforward computation now shows that 
$\cS_k \circ \cS_l = 
Q^{\sigma_{\min}(k+l) - \sigma_{\min}(k) \# \sigma_{\min}(l)} \cS_{k+l}$, 
where $\sigma_{\min}(k)$ denotes the minimal section class of $E_k$. 
The conclusion follows by setting $d(k,l) = \sigma_{\min}(k+l) 
- \sigma_{\min}(k) \# \sigma_{\min}(l)$ and 
the commutativity of $\#$. 
\end{proof} 

\subsection{Relation to the Seidel Representation} 
Taking the $z\to 0$ limit of shift operators, we obtain a 
big quantum cohomology version of the Seidel 
representation \cite{Seidel:pi1}. The author learned the idea of 
big Seidel elements from Eduardo Gonzalez during 
joint work \cite{Gonzalez-Iritani:Selecta} with him. 

\begin{definition}[Seidel elements] 
\label{def:Seidel_elem} 
Let $k\in \Hom(\Cstar, T)$ be a semi-negative homomorphism. 
The element $S_k(\tau) := \lim_{z\to 0} \bS_k(\tau) 1$ of 
$H_T^*(X)[\![Q]\!][\![\tau^0,\dots,\tau^m]\!]$ is called 
the \emph{Seidel element}.  
\end{definition}  

By Corollary \ref{cor:intertwine}, 
the $z\to 0$ limit of the operator $\bS_k(\tau)$ 
commutes with the quantum multiplication, and therefore 
coincides with the quantum multiplication by $S_k(\tau)$ 
(see also \cite[\S 8]{Maulik-Okounkov:qcoh_qgroup}). 
By Corollary \ref{cor:commuting}, we have 
\[
S_k(\tau) \star S_l(\tau) = Q^{d(k,l)} S_{k+l}(\tau).   
\]
This is called the \emph{Seidel representation}. 

\subsection{Relation to the $\hGamma$-Integral Structure} 
We remark a relationship between the shift operator 
and the $\hGamma$-integral structure introduced in 
\cite{Iritani, KKP, CIJ}. For quantum cohomology of 
the Hilbert scheme of points on $\C^2$, it has been observed 
in \cite{Okounkov-Pandharipande:Hilbert} that 
certain $\Gamma$-factors play an important role in the difference equation 
associated to the shift operators. 

We recall the $\hGamma$-class of $X$. 
Let $\delta_1,\dots,\delta_D$ denote the $T$-equivariant Chern roots 
of the tangent bundle $TX$ such that $c^T(TX) = (1 + \delta_1) 
\cdots (1+ \delta_D)$. The \emph{$T$-equivariant $\hGamma$-class} 
of $X$ is the class 
\[
\hGamma_X = \hGamma(TX) = \prod_{i=1}^D \Gamma (1 + \delta_i) 
\] 
in $H^{**}_T(X) = \prod_{p=0}^\infty H^p_T(X)$. 
Here $\Gamma(z) = \int_0^\infty e^{-t} t^{z-1} dt$ 
is Euler's $\Gamma$-function. By Taylor expansion, 
the right-hand side becomes a symmetric formal power series 
in $\delta_1,\dots,\delta_D$ and thus can be expressed in terms 
of the equivariant Chern classes of $TX$. 

The \emph{$\hGamma$-integral structure} 
assigns the following homogeneous flat section $\frs(E)$ of the quantum 
connection to a $T$-equivariant vector bundle $E\to X$: 
\[
\frs(E) = (2\pi)^{-D/2} M(\tau)^{-1} z^{-\mu}z^{c_1(X)} 
\hGamma_X (2\pi\iu)^{\deg/2}\ch^T(E) 
\]
where $D = \dim_\C X$, $M(\tau)$ is the fundamental solution in 
Proposition \ref{prop:fundsol}, $\mu\in \End_\C(H^*_T(X))$ 
is the Hodge grading operator defined by 
$\mu(\phi_i) = (\frac{\deg \phi_i}{2} - \frac{D}{2}) \phi_i$, 
$z^{c_1(X)} = e^{c_1(X) \log z}$ and $(2\pi \iu)^{\deg/2} \ch^T(E) 
 = \sum_{p=0}^\infty (2\pi\iu)^p \ch_p^T(E)$. 
The section $\frs(E)$ is flat, i.e.~$\nabla_i \frs(E) = 0$ and 
is homogeneous in the sense that 
\[
\left[ z \parfrac{}{z} + \mu + \sum_{i=0}^N 
\left(1-\frac{1}{2} \deg \phi_i \right) 
\tau^i \parfrac{}{\tau^i}  + \sum_{i=0}^N \rho^i \parfrac{}{\tau^i} 
\right] \frs(E) = 0 
\]
where we set $c_1(X) = \sum_{i=0}^N \rho^i \phi_i$. 
A key property of $\frs(E)$ is that the pairing 
\[
\left(\frs(E)(\tau,e^{-\pi\iu}z), \frs(F)(\tau,z)\right) 
\]
equals the $T$-equivariant Euler pairing $z^{-\deg/2} 
(2\pi\iu)^{\deg/2}\chi(E,F)$, where 
$\chi(E,F)= \sum_{i=0}^D (-1)^i \ch^T(\Ext^i(E,F)) 
\in H^{**}_T(\pt)$.  This follows from an appropriate 
equivariant Hirzebruch-Riemann-Roch formula. 
See \cite[\S 2-3]{CIJ} for more details. 

The $T$-equivariant $K$-group is a module over 
$K^0_T(\pt) = \C[T]$ and the Chern character 
$\ch^T \colon K^0_T(\pt) \to H^{**}_T(\pt)$ 
can be viewed as the pull-back by the universal covering 
$\exp \colon \Lie(T) = \C^m \to T= (\Cstar)^m$. 
A deck-transformation of this covering is given by 
the shift\footnote
{the shift by $2\pi\iu$ is superseded by the shift by $z$ because 
of the operators $z^{-\mu}$ and $(2\pi\iu)^{\deg/2}$.} 
of equivariant parameters $\lambda_j \to \lambda_j+2\pi\iu$. 
This suggests that $\frs(E)$ should be ``invariant'' 
under integral shifts of equivariant parameters. 
\begin{proposition} 
\label{prop:Gamma_integral_shift}
When the Novikov variable $Q$ is set to be one, 
the flat section $\frs(E)$ is invariant under the shift 
operator: 
\[
\bS_k \frs(E) = \frs(E) 
\]
for every semi-negative $k\in \Hom(\Cstar,T)$. 
\end{proposition} 
\begin{proof} 
As is discussed in \cite[\S 3]{CIJ}, the divisor equation 
shows that the specialization $Q=1$ of the Novikov variable 
is well-defined for $\frs(E)$. 
In view of the intertwining property in Theorem \ref{thm:intertwine}, 
it suffices to show that 
\[
\tcS_k \left( z^{-\mu} z^{c_1(X)} \hGamma_X 
(2\pi\iu)^{\deg/2} \ch(E) \right)= 
z^{-\mu} z^{c_1(X)} \hGamma_X 
(2\pi\iu)^{\deg/2} \ch(E).  
\] 
The restriction to the $T$-fixed component $F_i$ gives 
\begin{multline*} 
\left[ z^{-\mu} z^{c_1(X)} \hGamma_X (2\pi\iu)^{\deg/2}\ch(E) \right]_{F_i} 
= z^{D/2} z^{c_1(F_i)/z} \left(z^{-\deg/2} \hGamma_{F_i}\right) \\
\times \left( \prod_{\alpha} \prod_{j=1}^{\rank N_{\alpha,i}} 
z^{(\rho_{i,\alpha,j}+\alpha)/z}
\Gamma\left(1+ \frac{\rho_{i,\alpha,j}}{z} + \frac{\alpha}{z}\right) \right) 
\sum_{\epsilon} e^{2\pi\iu \epsilon/z} 
\end{multline*} 
where $\epsilon$ ranges over $T$-equivariant 
Chern roots of $E$ and we use the notation 
from Definition \ref{def:shift_Givental}. 
The conclusion easily follows from the identity $\Gamma(1+z) = z \Gamma(z)$. 
\end{proof} 

\section{Toric Mirror Theorem} 
\label{sec:mirrorthm} 
In this section we give a new proof of a mirror theorem 
\cite{Givental:toric_mirrorthm} 
for toric manifolds.  

\subsection{Toric Manifolds} 
We fix notation for toric manifolds. 
For background materials on toric manifolds, we refer the reader to 
\cite{Oda:toric,Audin,CLS}. 
Let $\bN \cong \Z^D$ 
denote a lattice. A toric manifold is given by a rational 
simplicial fan $\Sigma$ in the vector space $\bN_\R = 
\bN \otimes \R$. We assume that 
\begin{itemize} 
\item each cone $\sigma$ of $\Sigma$ is generated by 
part of a $\Z$-basis of $\bN$; 
\item the support $|\Sigma| = \bigcup_{\sigma\in \Sigma} \sigma$ 
of $\Sigma$ is convex and full-dimensional; 
\item $\Sigma$ admits a strictly convex piecewise linear 
function $\eta \colon |\Sigma| \to \R$. 
\end{itemize} 
These assumptions ensure that the corresponding toric variety 
$X_\Sigma$ is smooth and satisfies the hypotheses 
in \S\ref{subsec:hypotheses}. We do not require that $X$ 
is compact, or $c_1(X)$ is semipositive. 
Let $b_1,\dots,b_m\in \bN$ be primitive integral generators 
of one-dimensional cones of $\Sigma$. 
Let $\beta \colon \Z^m \to \bN$ be the homomorphism 
sending the standard basis vector $e_i\in \Z^m$ to $b_i$. 
The \emph{fan sequence} is the exact sequence 
\[
\begin{CD}
0 @>>> \LL @>>> \Z^m @>{\beta}>> \bN @>>> 0
\end{CD}
\]
with $\LL = \Ker(\beta)$. Set $K = \LL\otimes \Cstar$. 
The inclusion $\LL \hookrightarrow \Z^m$ induces 
the inclusion $K \hookrightarrow (\Cstar)^m$ of tori 
and defines a linear $K$-action on $\C^m$. 
The toric variety associated to $\Sigma$ is given by 
the GIT quotient 
\[
X_\Sigma = U/K, \qquad 
U = \C^m \setminus Z 
\]
where $Z\subset \C^m$ is the common zero set of monomials 
$z^I = z_{i_1} \cdots z_{i_k}$ with $I=\{i_1,\dots,i_k\}$ 
such that $\{b_i : 1\le i\le m, i\notin I\}$ spans a 
cone in $\Sigma$. We consider the $T$-action on $X_\Sigma$ 
indueced by the $T = (\Cstar)^m$-action on $\C^m$. 

Let $\lambda_i \in H^2_T(\pt)\cong \Lie(T)^*$ 
denote the class corresponding to the $i$th projection 
$T \to \Cstar$. We have 
\[
H_T^*(\pt) = \C[\lambda_1,\dots,\lambda_m]. 
\]
All the $T$-weights of $H^0(X_\Sigma,\cO)$ are contained 
in the cone $\sum_{i=1}^m \R_{\ge 0}(-\lambda_i)$ 
and therefore the condition \eqref{weight_cone} in \S \ref{subsec:hypotheses} 
is satisfied. A cocharacter $k\colon \Cstar \to T$ is semi-negative 
in the sense of Definition \ref{def:seminegative}
if $\lambda_i\cdot k \ge 0$ for all $i=1,\dots,m$. 

Let $u_i\in H^2_T(X_\Sigma)$ denote the class of the 
torus-invariant divisor $\{z_i = 0\}$ defined as the vanishing set 
of the $i$th co-ordinate $z_i$ on $\C^m$. 
The $T$-equivariant 
cohomology ring of $X_\Sigma$ is generated by 
these classes: 
\[
H^*_T(X_\Sigma) \cong H^*_T(\pt) [u_1,\dots,u_m]/ (\frI_1 + \frI_2)
\] 
where $\frI_1$ is the ideal generated by $\prod_{i\in I} u_i$ 
such that $\{b_i : i\in I\}$ does not span a cone in $\Sigma$ 
and $\frI_2$ is the ideal generated by $\sum_{i=1}^m 
\chi(b_i) (u_i-\lambda_i)$ with $\chi \in \Hom(\bN,\Z)$. 

\subsection{Mirror Theorem} 
\label{subsec:mirrorthm} 
Define a cohomology-valued hypergeometric series $I(y,z)$ 
by the formula: 
\[
I(y,z) = z e^{\sum_{i=1}^m u_i \log y_i/z} 
\sum_{d\in \Eff(X_\Sigma)} 
\left( \prod_{i=1}^m 
\frac{\prod_{c=-\infty}^{0} (u_i + cz)}
{\prod_{c=-\infty}^{u_i \cdot d} (u_i + cz)} 
\right) 
Q^d y_1^{u_1\cdot d} \cdots y_m^{u_m\cdot d}. 
\]
This formula defines an element of 
$H^*_\hT(X_\Sigma)_\loc[\![Q]\!][\![\log y]\!]$. 
We may write $I(y,z)$ as a sum over $H_2(X_\Sigma,\Z)$ since 
the summand automatically 
vanishes if $d\notin \Eff(X_\Sigma)$. 

Givental's mirror theorem \cite{Givental:toric_mirrorthm} 
(generalized later in \cite{LLY:III,Iritani:genmir,Brown:toric_fibration}) 
states the following: 
\begin{theorem} 
\label{thm:mirrorthm} 
The function $I(y,-z)$ lies on the Givental cone 
associated to genus-zero Gromov-Witten theory of $X_\Sigma$. 
\end{theorem} 

We explain the meaning of the statement. 
The \emph{Givental cone} $\cL$ \cite{Givental:symplectic} is a subset of 
$H_{\hT}^*(X_\Sigma)_\loc[\![Q]\!]$ consisting 
of points of the form: 
\begin{equation} 
\label{eq:cone_point}
-z + \bt(z) + \sum_{i=0}^N
\sum_{n=0}^\infty  \sum_{d\in \Eff(X_\Sigma)} 
\frac{Q^d}{n!} 
\corr{\frac{\phi^i}{-z-\psi},\bt(\psi),\dots,\bt(\psi)}_{0,n+1,d}^{X,T} \phi_i 
\end{equation} 
with $\bt(z) \in H_{\hT}^*(X_\Sigma)[\![Q]\!] 
= H_T^*(X_\Sigma)[z][\![Q]\!]$. 
The Givental cone $\cL$ can be written as the 
graph of the differential of the genus-zero descendant Gromov-Witten potential, and 
encodes all genus-zero descendant Gromov-Witten invariants. 
Theorem \ref{thm:mirrorthm} says that 
$I(y,z)$ is of the form \eqref{eq:cone_point} 
for some $\bt(z) \in H_T^*(X_\Sigma)[z][\![Q]\!][\![\log y]\!]$ 
with $\bt(z)|_{Q=\log y=0} = 0$. 
For toric manifolds, the above $I$-function 
determines the Givental cone and hence 
all the genus-zero Gromov-Witten invariants completely. 

In this paper, we use an alternative description \cite{Givental:symplectic} 
of the Givental cone $\cL$. We can write $\cL$ as the union 
\[
\cL = \bigcup_{\tau\in H^*_T(X_\Sigma)[\![Q]\!]} z T_\tau 
\]
of the semi-infinite subspaces $T_\tau = 
M(\tau,-z) H_T(X_\Sigma)[z][\![Q]\!]$, 
where $M(\tau,-z)$ denotes the fundamental solution from 
Proposition \ref{prop:fundsol} with the sign of $z$ flipped. 
The subspace $T_\tau$ is a (common) tangent space to $\cL$ 
along $z T_\tau \subset \cL$.  
Therefore, it suffices to show that $I(y,z)$ can be written 
in the form 
\[
I(y,z) = z M(\tau(y),z) \Upsilon(y,z)
\]
for some $\tau(y) \in H^*_T(X_\Sigma)[\![Q]\!][\![\log y]\!]$ 
and $\Upsilon(y,z) \in H^*_T(X_\Sigma)[z][\![Q]\!][\![\log y]\!]$. 

\subsection{Proof}
The idea of the proof is as follows. Let $e_i$ denote 
the cocharacter $\Cstar \to T = (\Cstar)^m$ given 
by the inclusion of the $i$th factor. Let $\bS_i = \bS_{e_i}$, 
$\cS_i = \cS_{e_i}$ denote the corresponding shift 
operators. 
In view of Theorem \ref{thm:intertwine}, the shift operator 
$\cS_i$ defines a vector field on the Givental cone:  
\begin{equation} 
\label{eq:vf_cone}
\cL \ni \bbf \longmapsto z^{-1} \cS_i \bbf \in T_{\bbf} \cL. 
\end{equation} 
These vector fields define commuting flows 
by Corollary \ref{cor:commuting}. 
We will identify the $I$-function with an integral submanifold 
of these vector fields.

Consider the $\Cstar$-action on $X_\Sigma$ 
induced by the cocharacter $e_i\in \Hom(\Cstar,T)$. 
The minimal fixed component $F_{\min}$ 
for this $\Cstar$-action is the toric divisor $\{z_i = 0\}$.  
Let $E_i=E_{e_i}$ denote the associated bundle. 
For a fixed point $x\in X_\Sigma^T$, we set $d_i(x) = \sigma_x - \sigma_{\min} 
\in H_2(X_\Sigma,\Z)$, where $\sigma_x \in H_2^{\sec}(E_k)$ is the 
section \eqref{eq:fixed_point_section} of $E_i$ 
associated to $x$ and $\sigma_{\min}\in H_2^{\sec}(E_k)$ 
is the minimal section class of $E_i$. 
We write $u_j(x)\in H^2_T(\pt)$ 
for the restriction of $u_j$ to $x$. 
\begin{lemma} 
\label{lem:fixed_weight} 
With the notation as above, we have 
\[
u_j(x) \cdot e_i = \delta_{ij} - u_j \cdot d_i(x). 
\]
\end{lemma} 
\begin{proof} Consider the $\hT$-invariant 
divisor $\{z_j=0\} \times \PP^1$ in $E_i$ 
and let $\hu_j$ denote the $\hT$-equivariant 
Poincar\'{e} dual of the divisor. 
Then we have $\hu_j|_{(x,[1,0])} = u_j(x)$ 
and $\hu_j|_{(x,[0,1])} = u_j(x) + (u_j(x) \cdot e_i) z$. 
The localization formula gives 
\[
\hu_j \cdot \sigma_x =\frac{\hu_j|_{(x,[1,0])}}{z} + 
\frac{\hu_j|_{(x,[0,1])}}{-z} 
= - u_j(x) \cdot e_i. 
\]
Similarly we have 
$\hu_j \cdot \sigma_{\min} = -u_j(y) \cdot e_i$ 
for any $T$-fixed point $y$ in the divisor $F_{\min} = 
\{z_i = 0\}$. If $i\neq j$, taking $y$ away from 
$\{z_j = 0\}$, we get $u_j(y) =0$. 
If $i=j$, $u_j(y)\cdot e_i = 1$. 
Therefore $\hu_j \cdot \sigma_{\min}= - \delta_{ij}$. 
The conclusion follows. 
\end{proof}

\begin{lemma} 
\label{lem:I_flow} 
The $I$-function is an integral curve of the vector field \eqref{eq:vf_cone}, 
that is, for $i\in \{1,\dots,m\}$, we have 
\[
z\parfrac{}{y_i} I(y,z) = \cS_i I(y,z). 
\]
\end{lemma} 
\begin{proof} 
Note that all the $T$-fixed points on $X_\Sigma$ 
are isolated.  
Let $x\in X^T$ be a fixed point. It suffices to show that 
\[
z\parfrac{}{y_i} I_x(y,z) = \Delta_x(e_i) e^{- z\partial_{\lambda_i}} 
I_x(y,z) 
\]
where $I_x(y,z)$ is the restriction of the $I$-function to $x$ and 
\[
\Delta_x(e_i) = Q^{d_i(x)} \prod_{j=1}^m 
\frac{\prod_{c=-\infty}^0 (u_j(x) + cz )}
{\prod_{c=-\infty}^{-u_j(x) \cdot e_i} (u_j(x) + cz)}. 
\]
Using Lemma \ref{lem:fixed_weight}, we have 
\begin{multline*} 
\Delta_x(e_i) e^{-z \partial_{\lambda_i}} I_x(y,z) 
 = z e^{\sum_{j=1}^m u_j(x) \log y_j/z} 
e^{-\log y_i + \sum_{j=1}^m (u_j \cdot d_i(x))\log y_j } \\ 
\times  Q^{d_i(x)} 
\sum_{d\in H_2(X_\Sigma,\Z)} 
\left( \prod_{j=1}^m 
\frac{\prod_{c=-\infty}^0 (u_j(x) + cz )}
{\prod_{c=-\infty}^{-u_j(x) \cdot e_i} (u_j(x) + cz)} 
\frac{\prod_{c=-\infty}^{-u_j(x) \cdot e_i} (u_j(x) + cz)}
{\prod_{c=-\infty}^{u_j\cdot d - u_j(x) \cdot e_i} (u_j(x) + cz)} 
\right) 
Q^d y^d 
\end{multline*}
where $y^d = \prod_{j=1}^m y_j^{u_j \cdot d}$.  
Changing variables $d \to d - d_i(x)$ and using 
again Lemma \ref{lem:fixed_weight}, we find that 
this equals $z \parfrac{}{y_i} I(y,z)$. 
\end{proof} 

We identify the classical shift operators: 
\begin{notation} 
We set $v_i := u_i - \lambda_i \in H^2_T(X_\Sigma)$ 
and write $v_i(x)\in H^2_T(\pt)$ for the restriction of $v_i$ 
to a $T$-fixed point $x$. 
\end{notation} 
\begin{lemma} 
\label{lem:classical_shift}
Let $f(v,\lambda)$ 
be a cohomology class in $H^*_T(X_\Sigma)$ 
expressed as a polynomial in $v_1,\dots,v_m$ and 
$\lambda_1,\dots,\lambda_m$. 
When we write $\tau \in H^*_T(X_\Sigma)$ as a polynomial 
$\tau(v,\lambda)$ in $v_1,\dots,v_m$ and 
$\lambda_1,\dots,\lambda_m$, we have 
\[
\lim_{Q\to 0}\bS_i(\tau) f(v,\lambda) = u_i 
e^{(\tau(v,\lambda- e_i z) - \tau(v,\lambda))/z}
f(v, \lambda -z e_i)
\] 
where $\lambda - z e_i = (\lambda_1,\dots,\lambda_{i-1},\lambda_i - z, 
\lambda_{i+1},\dots,\lambda_m)$. In particular 
the classical Seidel elements are given by: 
\[
\lim_{Q\to 0} S_i(\tau) = u_i e^{-\parfrac{\tau(v,\lambda)}{\lambda_i}}. 
\]
\end{lemma} 
\begin{proof} 
Recall from Theorem \ref{thm:intertwine} 
that we have $\cS_i \circ M(\tau) = M(\tau) \circ \bS_i(\tau)$. 
Since $\lim_{Q \to 0} M(\tau) = e^{\tau/z}$, we have 
\[
\lim_{Q \to 0} \bS_i(\tau) f(v,\lambda)= 
e^{-\tau/z} 
\left(\lim_{Q \to 0}\cS_i \right) e^{\tau/z} f(v,\lambda). 
\]
By definition of $\cS_i$, this vanishes when restricted 
to a fixed point outside of the minimal fixed component 
$\{z_i=0\}$ with respect to $e_i$. 
On the other hand, for any $T$-fixed point $x$ in $\{z_i=0\}$, 
Lemma \ref{lem:fixed_weight} implies that 
$u_j(x) \cdot e_i = \delta_{ij}$, 
$v_j(x) \cdot e_i = 0$ and thus: 
\begin{align*} 
\lim_{Q\to 0} \bS_i(\tau) f(v,\lambda)\Bigr|_{x} 
& = e^{-\tau(v(x),\lambda)/z}u_i(x) e^{-z \partial_{\lambda_i}} 
\left[ 
e^{\tau(v(x),\lambda)/z} f(v(x),\lambda) \right] \\
& = u_i(x) e^{(\tau(v(x),\lambda-e_i z) - \tau(v(x),\lambda))/z }
f(v(x) , \lambda-e_i z) 
\end{align*} 
where we set $v(x) = (v_1(x),\dots,v_m(x))$. The conclusion 
follows. 
\end{proof} 

\begin{lemma} 
\label{lem:restriction_linearcomb} 
Let $x$ be a $T$-fixed point on $X_\Sigma$. The restriction $u_j(x)$ 
is a linear combination of $\lambda_i$ such that $x$ does not 
lie on the divisor $\{z_i=0\}$. 
\end{lemma}
\begin{proof} 
Note that if $x$ does not lie on the divisor $\{z_i=0\}$, we have 
$u_i(x) = 0$ and thus $v_i(x) = -\lambda_i$. This together with 
the linear relation $\sum_{i=1}^m \chi(b_i) v_i =0$, $\chi \in \Hom(\bN,\Z)$ 
determines $v_1(x),\dots,v_m(x)$ uniquely. This implies the conclusion. 
\end{proof} 

Let $\ovcL= \cL|_{z\to -z}$ denote the Givental cone with the sign of $z$ flipped. 
By the description in \S \ref{subsec:mirrorthm}, we have 
a parametrization of the Givental cone $\ovcL$ 
by $(\tau,\Upsilon) \in H^*_T(X) \times H^*_\hT(X) 
= H^*_T(X)\times H^*_T(X)[z]$ as: 
\[
(\tau, \Upsilon) \longmapsto z M(\tau,z) \Upsilon \in \ovcL. 
\]
The vector field \eqref{eq:vf_cone} on $\ovcL$ corresponds to 
the following vector field on $H^*_T(X) \times H^*_T(X)[z]$: 
\[
(\bV_i)_{\tau,\Upsilon} = (S_i(\tau), [z^{-1} \bS_i(\tau)]_+ \Upsilon)  
\]
where $S_i(\tau)$ is the Seidel element in Definition \ref{def:Seidel_elem} 
and $[\cdots]_+$ means the projection to the polynomial part in $z$, 
i.e.~$[z^{-1} \bS_i(\tau)]_+ \Upsilon 
= z^{-1} \bS_i(\tau) \Upsilon - z^{-1} S_i(\tau)\star_\tau\Upsilon$. 
In fact, if we have a curve $t \mapsto (\tau(t), \Upsilon(t))$ with 
$\tau'(0) = S_i(\tau(0))$, $\Upsilon'(0) = [z^{-1}\bS_i(\tau(0))]_+ \Upsilon(0)$, 
the corresponding curve $\bbf(t)=z M(\tau(t),z) \Upsilon(t)$ on $\ovcL$ 
satisfies 
\begin{align*} 
\bbf'(0) & = 
M(\tau(0),z) \left( S_i(\tau(0)) \star_{\tau(0)}\Upsilon(0) \right) + 
z M(\tau(0),z) [z^{-1} \bS_i(\tau(0))]_+ \Upsilon(0)  \\ 
& = M(\tau(0),z) \bS_i(\tau(0)) \Upsilon(0) 
 = z^{-1} \cS_i \bbf(0) 
\end{align*} 
where we used $z \partial_i M(\tau,z) = M(\tau,z) (\phi_i\star_\tau)$ 
in the first line and Theorem \ref{thm:intertwine} in the second line. 
Since the vector fields \eqref{eq:vf_cone} commute each other, 
the corresponding vector fields $\bV_i$, $i=1,\dots,m$ 
also commute each other.  
In what follows, we show the existence of an integral curve 
for the vector field $\bV_i$ with prescribed asymptotics. 

\begin{proposition} 
\label{prop:tau_Upsilon} 
There exist unique functions 
\[
\tau(y) \in H^*_T(X_\Sigma)[\![Q]\!][\![\log y]\!] 
\quad \text{and} \quad 
\Upsilon(y,z) \in H^*_T(X_\Sigma)[z][\![Q]\!][\![\log y]\!]
\] 
which are of the form 
\begin{align*} 
\tau(y) &= \sum_{i=1}^m u_i \log y_i + 
\sum_{d\in \Eff(X_\Sigma), d\neq 0} 
Q^d y^d \tau_d \\ 
\Upsilon(y,z) & = 1 + \sum_{d\in \Eff(X_\Sigma), d\neq 0} 
Q^d y^d \Upsilon_d 
\end{align*} 
with $y^d = \prod_{j=1}^m y_j^{u_j \cdot d}$ and 
give an integral curve for the vector field $\bV_i$: 
\[
\parfrac{\tau(y)}{y_i} = S_i(\tau(y)) \quad \text{and} \quad 
\parfrac{\Upsilon(y,z)}{y_i} = 
\left[z^{-1} \bS_i(\tau(y))\right]_+ \Upsilon(y,z)  
\]
for all $1\le i\le m$. 
\end{proposition} 
\begin{proof} 
Write $\tau(y) = \sum_{j=1}^m u_j \log y_j + \tau'$. 
The divisor equation in Remark \ref{rem:div_shift} gives 
\[
\bS_i(\tau(y)) = y_i^{-1} \bS_i(\tau';Qy). 
\]
where $\bS_i(\sigma;Qy)$ is obtained from $\bS_i(\sigma)$ by 
replacing $Q^d$ with $Q^dy^d$. 
Therefore we need to solve the following differential equations: 
\begin{equation} 
\label{eq:flow_aftersubtractinglog}
y_i \parfrac{\tau'}{y_i} = S_i(\tau'; Qy) - u_i \quad 
\text{and} \quad 
y_i \parfrac{\Upsilon}{y_i} = \left[z^{-1} \bS_i(\tau'; Qy)  \right]_+\Upsilon. 
\end{equation} 
We expand 
\[
\tau' = \sum_{d\in \Eff(X_\Sigma), d\neq 0} \ttau_d(y) Q^d, \qquad 
\Upsilon = \sum_{d\in \Eff(X_\Sigma)} \tUpsilon_d(y) Q^d 
\]
with $\tUpsilon_0(y) =1$ 
and solve for the coefficients $\ttau_d(y)$, $\tUpsilon_d(y)$ 
recursively. Note that the equation \eqref{eq:flow_aftersubtractinglog} 
holds true mod $Q$ by Lemma \ref{lem:classical_shift}. 

First we solve for $\tau'$. Choose a K\"{a}hler class $\omega$ 
such that $\omega \cdot d_1 = \omega \cdot d_2$ for 
$d_1,d_2\in \Eff(X_\Sigma)$ 
if and only if $d_1 = d_2$. This defines a positive real grading on the Novikov ring 
$\C[\![Q]\!]$ such that $\deg Q^d = \omega \cdot d$. 
Take $d_0 \in \Eff(X_\Sigma) \setminus\{0\}$. 
Suppose by induction that there exist 
$\ttau_d$ for all $d$ with $\omega \cdot d < \omega \cdot d_0$ 
such that $\ttau_d=\tau_d y^d$ for some $\tau_d \in H^*_T(X)$ 
and that $\tau' = \sum_{\omega \cdot d <\omega\cdot d_0} \ttau_d Q^d$ 
satisfies the differential equation \eqref{eq:flow_aftersubtractinglog} 
modulo terms of degree $\ge \omega \cdot d_0$. 
We write $\tau_d$ as a polynomial in $v_1,\dots,v_m$ 
and $\lambda_1,\dots,\lambda_m$. 
Comparing the coefficients of $Q^{d_0}$ 
of the differential equation, we obtain 
using Lemma \ref{lem:classical_shift} that: 
\[
y_i \parfrac{\ttau_{d_0}}{y_i}+ u_i \parfrac{\ttau_{d_0}}{\lambda_i} 
= \left(\begin{array}{l} 
\text{an expression in $\ttau_d$} \\ 
\text{with $\omega \cdot d<\omega \cdot d_0$} 
\end{array}\right). 
\]
Here the right-hand side is of the form  
$g_i(v,\lambda) y^{d_0}$ by induction hypothesis, 
where $g_i(v,\lambda)$ is a polynomial in $v_1,\dots,v_m$ 
and $\lambda_1,\dots,\lambda_m$. Setting $\ttau_{d_0} 
= \tau_{d_0} y^{d_0}$, we obtain 
\[
(u_i \cdot d_0) \tau_{d_0} + (v_i + \lambda_i) 
\parfrac{\tau_{d_0}}{\lambda_i} = g_i(v,\lambda). 
\]
The K\"{a}hler class can be written as a non-negative linear 
combination of $u_i$, and thus there exists $i_0$ such that 
$u_{i_0} \cdot d_0>0$. Then we can solve for 
the polynomial $\tau_{d_0}=\tau_{d_0}(v,\lambda)$ 
from the above equation with $i=i_0$ 
recursively from the highest order term in $\lambda_{i_0}$. 
Setting $\tau(y)= \sum_i u_i \log y_i + 
\sum_{\omega \cdot d\le \omega \cdot d_0} \tau_d y^d Q^d$, 
we have 
\[
\parfrac{\tau(y)}{y_i} \equiv S_i(\tau(y))
\]
modulo terms of degree $\ge \omega\cdot d_0$ 
for $i\neq i_0$ and modulo terms of degree $> \omega \cdot d_0$ 
for $i=i_0$. The commutativity of the flow implies that we have 
for $i\neq i_0$, 
\begin{align}
\label{eq:commutativity_check}
\begin{split} 
\parfrac{}{y_{i_0}} \left(\parfrac{\tau}{y_i} - S_i(\tau(y))\right) 
& = \parfrac{^2 \tau(y)}{y_i \partial y_{i_0}} 
- (d_{\parfrac{\tau(y)}{y_{i_0}}}S_i)(\tau(y)) \\ 
& \equiv \parfrac{S_{i_0}(\tau(y))}{y_i} 
- (d_{S_{i_0}(\tau(y))}S_i) (\tau(y)) \\ 
& = (d_{\parfrac{\tau(y)}{y_i}}S_{i_0})(\tau(y))  
- (d_{S_i (\tau(y))}S_{i_0}) (\tau(y))  \\ 
& = (d_{\parfrac{\tau(y)}{y_i} - S_i(\tau(y))} S_{i_0})(\tau(y))
\end{split} 
\end{align} 
modulo terms of degree $> \omega \cdot d_0$. 
Using the divisor equation again, we have 
\[
y_i \left(\parfrac{\tau(y)}{y_i} - S_i(\tau(y))\right) = 
u_i + y_i \parfrac{\tau'}{y_i} - S_i(\tau';Qy). 
\]
Modulo terms of degree $>\omega \cdot d_0$, 
this is $\alpha (Qy)^{d_0}$ for some $\alpha = \alpha(v,\lambda) 
\in H^*_T(X)$. 
Now the coefficient of $Q^{d_0}$ of 
equation \eqref{eq:commutativity_check} gives 
(by Lemma \ref{lem:classical_shift}): 
\[
(u_{i_0} \cdot d_0) \alpha + u_{i_0} \parfrac{\alpha}{\lambda_{i_0}} = 0.  
\]
We want to show that $\alpha=0$ as a cohomology class. 
Consider the restriction $\alpha(x)$ of $\alpha$ to a $T$-fixed point 
$x\in X_\Sigma$. If $x$ lies in the divisor $\{z_{i_0}=0\}$, 
$v_j(x)\in H_T^2(\pt)$ is a linear combination of $\lambda_{j'}$ 
with $j' \neq i_0$ by Lemma \ref{lem:restriction_linearcomb}. 
Thus 
\begin{equation} 
\label{eq:restriction_differentiation}
\left.\parfrac{\alpha}{\lambda_{i_0}}\right|_{x} = 
\parfrac{\alpha(x)}{\lambda_{i_0}}. 
\end{equation} 
If $x$ is not in the divisor $\{z_{i_0}=0\}$, $u_{i_0}(x) = 0$. 
Therefore, by restricting to $x$, we have  
\[
(u_{i_0} \cdot d) \alpha(x) + u_{i_0}(x) 
\parfrac{\alpha(x)}{\lambda_{i_0}} = 0.  
\]
This shows that $\alpha(x) =0$ recursively 
from the highest order term in $\lambda_{i_0}$. 
Note that the same argument shows the uniqueness of $\tau_{d_0}$. 
This completes the induction. 

Next we solve for $\Upsilon$ assuming that $\tau'$ is already solved.  
Let $\omega$ be a K\"{a}hler class as above and $d_0\in \Eff(X_\Sigma)$ 
be a non-zero effective class. Suppose by induction that 
there exist $\tUpsilon_d$ for all $d$ with $\omega \cdot d 
< \omega \cdot d_0$ such that $\tUpsilon_d = \Upsilon_d y^d$ 
and that $\Upsilon = \sum_{\omega \cdot d < \omega \cdot d_0} 
\tUpsilon_d Q^d$ satisfies the differential equation 
\eqref{eq:flow_aftersubtractinglog} modulo terms of 
degree $\ge \omega\cdot d_0$. We regard $\Upsilon_d$ 
as a polynomial in $v_1,\dots,v_m$ and $\lambda_1,\dots,\lambda_m$. 
Comparing the coefficients of $Q^{d_0}$ of the differential equation 
and using Lemma \ref{lem:classical_shift}, 
we obtain 
\[
y_i \parfrac{\tUpsilon_{d_0}(v,\lambda)}{y_i} 
-  (v_i + \lambda_i) z^{-1}\left(\tUpsilon_{d_0}(v, \lambda -e_i z) 
- \tUpsilon_{d_0}(v,\lambda) \right) 
= \left(\begin{array}{l} 
\text{an expression in $\tUpsilon_d$} \\ 
\text{with $\omega \cdot d<\omega \cdot d_0$} 
\end{array}\right). 
\]
Here the right-hand side is of the form $g_i(v,\lambda) y^{d_0}$ 
for some polynomial $g_i(v,\lambda)$ in $v_1,\dots,v_m$ 
and $\lambda_1,\dots,\lambda_m$. 
Setting $\tUpsilon_{d_0} = \Upsilon_{d_0} y^{d_0}$, we have 
\[
(u_i \cdot d_0) \Upsilon_{d_0}(v,\lambda)  - 
 (v_i + \lambda_i) z^{-1}\left( \Upsilon_{d_0}(v,\lambda - e_i z) 
- \Upsilon_{d_0}(v,\lambda) \right)
= g_i(v,\lambda). 
\]
As before, we can find $i_0$ such that $u_{i_0} \cdot d_0 >0$. 
We can solve for $\Upsilon_{d_0}(v,\lambda)$ recursively 
from the highest order term in $\lambda_{i_0}$ using 
this equation with $i=i_0$. 
Setting $\Upsilon = \sum_{\omega \cdot d \le \omega \cdot d_0} 
\Upsilon_d Q^d$, we have 
\[
\parfrac{\Upsilon(y)}{y_i} \equiv \left[z^{-1} \bS_i(\tau(y)) \right]_+ \Upsilon(y) 
\]
modulo terms of degree $\ge \omega \cdot d_0$ for $i\neq i_0$ 
and modulo terms of degree $> \omega\cdot d_0$ for $i=i_0$. 
We have for $i\neq i_0$, 
\begin{align*} 
\parfrac{}{y_{i_0}} &
 \left(\parfrac{\Upsilon(y)}{y_i}  - [z^{-1} \bS_i(\tau(y))]_+ \Upsilon(y)\right) 
= \parfrac{^2 \Upsilon(y)}{y_i \partial y_{i_0}} 
-  \parfrac{}{y_{i_0}}[z^{-1}\bS_i(\tau(y)) ]_+\Upsilon(y) \\ 
& \quad \equiv \parfrac{}{y_i} [z^{-1} \bS_{i_0}(\tau(y))]_+ \Upsilon(y) 
- \parfrac{}{y_{i_0}}[z^{-1}\bS_i(\tau(y))]_+ \Upsilon(y) \\ 
& \quad \equiv 
\left[ z^{-1} (d_{S_i(\tau(y))} \bS_{i_0})(\tau(y)) \right]_+
\Upsilon(y) + [z^{-1} \bS_{i_0}(\tau(y))]_+ \parfrac{\Upsilon(y)}{y_{i}} 
 \\ 
& \quad \qquad - \left[ z^{-1} (d_{S_{i_0}(\tau(y))} \bS_i)(\tau(y)) \right]_+ 
\Upsilon(y) 
- [z^{-1} \bS_i(\tau(y))]_+ [z^{-1} \bS_{i_0}(\tau(y))]_+\Upsilon(y)
\end{align*} 
modulo terms of degree $> \omega \cdot d_0$. 
The commutativity of the flows $\bV_i$, $i=1,\dots,m$ implies for $i\neq j$, 
\begin{multline*} 
\left[z^{-1} (d_{S_i(\tau)} \bS_{j})(\tau) \right]_+ 
\Upsilon + [z^{-1} \bS_j(\tau)]_+ [z^{-1} \bS_i(\tau) ]_+\Upsilon \\ 
= \left[z^{-1} (d_{S_j(\tau)}\bS_i)(\tau)\right]_+ \Upsilon 
+ [z^{-1} \bS_i(\tau)]_+ [z^{-1} \bS_j(\tau) ]_+\Upsilon. 
\end{multline*} 
Therefore we have: 
\begin{equation} 
\label{eq:commutativity_check2}
\parfrac{}{y_{i_0}} 
 \left(\parfrac{\Upsilon(y)}{y_i}  - [z^{-1} \bS_i(\tau(y))]_+ \Upsilon(y)\right) 
\equiv [z^{-1} \bS_{i_0}(\tau(y)) ]_+ 
\left( \parfrac{\Upsilon(y)}{y_i} -  
[z^{-1} S_i(\tau(y))]_+  \Upsilon(y)
\right)
\end{equation} 
modulo terms of degree $>\omega\cdot d_0$. 
By the divisor equation, we have 
\[
y_i \left(\parfrac{\Upsilon(y)}{y_i} - [z^{-1} \bS_i(\tau(y))]_+\Upsilon(y) \right) 
= y_i \parfrac{\Upsilon(y)}{y_i} - [z^{-1} \bS_i(\tau';Qy)]_+ \Upsilon(y). 
\]
This is of the form $\alpha (Qy)^{d_0}$ for some $\alpha =\alpha(v,\lambda,z) 
\in H^*_\hT(X_\Sigma)$, modulo terms of degree $>\omega \cdot d_0$. 
Hence the differential equation \eqref{eq:commutativity_check2} implies 
via Lemma \ref{lem:classical_shift} that: 
\[
(u_{i_0} \cdot d_0) \alpha - 
u_{i_0} z^{-1}(\alpha(v,\lambda-e_{i_0} z, z) - \alpha(v,\lambda,z)) = 0. 
\]
We want to show that $\alpha=0$ in the cohomology group. 
By restricting this to a $T$-fixed point $x$ and using a similar 
argument as before (see \eqref{eq:restriction_differentiation}), 
we obtain 
\[
(u_{i_0} \cdot d_0) \alpha(x) - (v_{i_0}(x) + \lambda_{i_0}) 
z^{-1} \left(e^{-z \partial_{\lambda_{i_0}}}\alpha(x)  - \alpha(x)\right) = 0  
\]
for the restriction $\alpha(x)\in H^*_\hT(\pt)$ of $\alpha$ 
to $x$. 
We can easily see that $\alpha(x) = 0$ recursively 
from the highest order term in $\lambda_{i_0}$. 
Therefore $\alpha=0$. Note that the same argument also 
shows the uniqueness of $\Upsilon_{d_0}$. 
This completes the induction and the proof. 
\end{proof} 



We now come to the final step of the proof. Let $\tau(y)$, $\Upsilon(y,z)$ 
be as in Proposition \ref{prop:tau_Upsilon}. Then, as discussed 
in a paragraph preceding Proposition \ref{prop:tau_Upsilon}, 
\[
y \longmapsto \bbf(y) := z M(\tau(y),z) \Upsilon(y,z) 
\]
defines an integral manifold for the vector fields in \eqref{eq:vf_cone}. 
We shall show that $\bbf(y) = I(y,z)$. 
Using the divisor equation for $M(\tau,z)$, 
we find that $\bbf(y)$ is of the form: 
\begin{equation} 
\label{eq:asympt_I} 
\bbf(y) = z e^{\sum_{i=1}^m u_i \log y_i/z} 
\left( 
1 + \sum_{d \in \Eff(X_\Sigma)\setminus \{0\}} \bbf_d Q^d y^d 
\right) 
\end{equation} 
with $\bbf_d \in H_{\hT}(X)_\loc$. 
In view of Lemma \ref{lem:I_flow}, the following lemma 
shows that $\bbf(y) = I(y,z)$ and 
completes the proof of Theorem \ref{thm:mirrorthm}. 
\begin{lemma} 
The family of elements $y\mapsto \bbf(y)$ of the form \eqref{eq:asympt_I} 
satisfying $\partial_{y_i} \bbf(y) = z^{-1} \cS_i \bbf(y)$, $i=1,\dots,m$ 
is unique. 
\end{lemma} 
\begin{proof} 
Suppose that we have two families $\bbf_1(y)$, $\bbf_2(y)$ of elements 
of the form \eqref{eq:asympt_I} satisfying 
$\partial_{y_i} \bbf_j(y) = z^{-1} \cS_i \bbf_j(y)$, $j=1,2$, 
$i=1,2,\dots,m$. 
The difference $g(y) = \bbf_1(y) - \bbf_2(y)$ satisfies the same 
differential equation and is of the form 
\[
g(y) = z e^{\sum_{i=1}^m u_i \log y_i/z} \sum_{d\in \Eff(X_\Sigma) 
\setminus \{0\}} g_d Q^d y^d. 
\]
Choose a K\"{a}hler class $\omega$ and suppose by induction 
that we know $g_d =0$ for all $d \in \Eff(X_\Sigma)$ 
with $\omega \cdot d < \omega \cdot d_0$ for some 
$d_0 \in \Eff(X_\Sigma)\setminus \{0\}$. 
Let $x$ be a $T$-fixed point. Let $\delta$ be the set of 
indices $i$ such that $x$ does not lie on the toric 
divisor $\{z_i =0\}$. The K\"{a}hler class $\omega$ can be 
written as a positive linear combination of non-equivariant 
limits of $u_i$ with $i\in \delta$. Therefore, there exists 
$i_0\in \delta$ such that $u_{i_0} \cdot d_0>0$. 
The coefficient in front of $Q^{d_0}$ of the equation 
$\partial_{y_{i_0}} g (y)= z^{-1} \cS_{i_0} g(y)$ 
restricted to the fixed point $x$ 
gives: 
\[
(u_{i_0} \cdot d_0) g_{d_0}(x) = 0 
\]
since $x$ does not lie on the minimal fixed component 
$\{z_{i_0}=0\}$ with respect to $e_{i_0}$. 
Therefore $g_{d_0}(x)=0$. Since $x$ is arbitrary, $g_{d_0} = 0$. 
This completes the induction and the proof. 
\end{proof} 

\subsection{Example} 
Consider the toric variety $X_\Sigma = \PP^{m-1}$. 
In this case we have $m$ shift operators $\bS_1,\dots,\bS_m$ 
corresponding to $m$ toric divisors. It is well-known that 
the mirror map $\tau(y)$ and the function $\Upsilon(y)$ are trivial: 
\[
\tau(y) = \sum_{i=1}^m u_i \log y_i, \qquad 
\Upsilon(y) = 1. 
\]
Generalizing the differential equation in Lemma \ref{lem:I_flow}, 
we can show that 
\[
\cS_{i_1} \cdots \cS_{i_a} I(y,z)= 
z \partial_{y_{i_1}}\cdots z\partial_{y_{i_a}} I(y,z) 
\]
when $i_1,\dots,i_a$ are distinct. 
This together with the intertwining property $\cS_i \circ M(\tau,z) 
= M(\tau,z) \circ \bS_i(\tau)$ and the divisor equation 
$\bS_i(\tau(y)) = y_i^{-1}\bS_i(0;Qy)$ implies: 
\[
\bS_{i_1}(0;Qy) \cdots \bS_{i_a} (0;Qy) 1 = z \nabla_{u_{i_1}} 
\cdots z \nabla_{u_{i_a}} 1 \Bigr|_{\tau(y)} 
= 
\begin{cases} 
u_{i_1} \cdots u_{i_a} & \text{if $a<m$;}\\ 
Qy_1\cdots y_m & \text{if $a=m$},  
\end{cases} 
\]
where $i_1,\dots,i_a$ are distinct and 
$\bS_i(0;Qy)$ means $\bS_i(0)|_{Q \to Qy_1\cdots y_m}$. 
This determines the action of $\bS_i(0)$ completely. 
Since the one-parameter subgroup $e_1+ \cdots + e_m$ acts 
on $\PP^{m-1}$ trivially, we have a relation 
$\bS_1(\tau) \circ \cdots \circ \bS_m(\tau)=Q$ 
by Corollary \ref{cor:commuting}. 
Writing $u_i = v+ \lambda_i$ for $i=1,\dots,m$, we recover the relation: 
\[
(z\nabla_{v} +\lambda_1) \cdots (z\nabla_v + \lambda_m) 1 
\Bigr|_{\tau=0}= Q 
\]
in the equivariant small quantum $D$-module of $\PP^{m-1}$. 

\subsection{Remarks} 
We first remark a relation to the results in 
\cite{Gonzalez-Iritani:Selecta}. Let $X_\Sigma$ be 
a compact toric manifold such that $c_1(X_\Sigma)$ 
is nef. In this case, the mirror map $\tau(y)$ takes 
values in $H^2_T(X)$. We write 
\[
\tau(y) = \sum_{i=1}^m (\log y_i  - g^i(y)) u_i
\]
for some $\C$-valued functions $g^{i}(y)$.  
Using the divisor equation from Remark \ref{rem:div_shift}, the differential 
equation in Proposition \ref{prop:tau_Upsilon} implies: 
\[
y_i \parfrac{\tau(y)}{y_i} = e^{g^{i}(y)} S_i(0; Q e^{\tau(y)}) 
\]
where we set $S_i(0; Q e^{\tau(y)}) = S_i(0)|_{Q \to Q e^{\tau(y)}}$. 
The left-hand side is called the Batyrev element in 
\cite{Gonzalez-Iritani:Selecta} and this 
recovers the relationship between the Seidel and the Batyrev 
elements in \cite[Theorem 1.1]{Gonzalez-Iritani:Selecta}. 

We should also recover a mirror theorem for the 
extended $I$-function \cite{CCIT:mirrorthm} by considering 
the shift operators corresponding to general semi-negative 
cocharacters $k\in (\Z_{\ge 0})^m \subset \Hom(\C^\times ,T)$. 
It would be also interesting to see if our method can be generalized 
to toric orbifolds \cite{CCIT:mirrorthm,Cheong-CF-Kim}, 
toric fibrations \cite{Brown:toric_fibration}, 
or other $T$-varieties.

\bibliographystyle{alpha}
\bibliography{shift_toric}
\end{document}